\documentclass[11pt]{amsart}
\usepackage{amsthm}

\usepackage[all]{xy}
\usepackage{amssymb}
\usepackage{enumerate}
\usepackage{mathrsfs}
\usepackage{epsfig}
\usepackage{graphicx}
\usepackage{subfig}
\usepackage{float}
\usepackage{epigraph}

\numberwithin{equation}{section}

\newtheorem{thm}{Theorem}[section]
\newtheorem{cor}[thm]{Corollary}
\newtheorem{lem}[thm]{Lemma}
\newtheorem{prop}[thm]{Proposition}
\newtheorem{rem}{Remark}[section]
\newtheorem{example}[thm]{Example}

\newcommand{\eq}[1]{(\ref{#1})}

\renewcommand{\Re}{\operatorname{\rm Re}}
\renewcommand{\Im}{\operatorname{\rm Im}}

\newcommand{\beqast}{\begin{eqnarray*}}
\newcommand{\eqast}{\end{eqnarray*}}
\newcommand{\beqa}{\begin{eqnarray}}
\newcommand{\eqa}{\end{eqnarray}}

\newcommand{\bbe}{\begin{equation}}
\newcommand{\ee}{\end{equation}}

\renewcommand{\Re}{\operatorname{\rm Re}}
\renewcommand{\Im}{\operatorname{\rm Im}}

\newcommand{\bC}{{\mathbb C}}
\newcommand{\bE}{{\mathbb E}}
\newcommand{\bN}{{\mathbb N}}
\newcommand{\bP}{{\mathbb P}}
\newcommand{\bQ}{{\mathbb Q}}

\newcommand{\bR}{{\mathbb R}}

\newcommand{\bZ}{{\mathbb Z}}

\newcommand{\cF}{{\mathcal F}}

\newcommand{\cE}{{\mathcal E}}

\newcommand{\cL}{{\mathcal L}}

\newcommand{\cC}{{\mathcal C}}

\newcommand{\barX}{{\bar X}}
\newcommand{\uX}{{\underline X}}

\newcommand{\cEq}{{\mathcal E_q}}
\newcommand{\cEpq}{{\mathcal E^+_q}}
\newcommand{\cEmq}{{\mathcal E^-_q}}

\newcommand{\phipq}{{\phi^+_q}}
\newcommand{\phimq}{{\phi^-_q}}

\newcommand{\tV}{{\tilde V}}
\newcommand{\tW}{{\tilde W}}

\newcommand{\hf}{{\hat f}}
\newcommand{\hg}{{\hat g}}

\newcommand{\Om}{{\Omega}}

\newcommand{\al}{\alpha}
\newcommand{\alp}{\alpha_+}
\newcommand{\alm}{\alpha_-}
\newcommand{\be}{\beta}
\newcommand{\De}{\Delta}
\newcommand{\de}{\delta}
\newcommand{\eps}{\epsilon}

\newcommand{\la}{\lambda}

\newcommand{\mumpr}{\mu'_-}
\newcommand{\muppr}{\mu'_+}

\newcommand{\sg}{\sigma}

\newcommand{\om}{\omega}
\newcommand{\omm}{\om_-}
\newcommand{\omp}{\om_+}

\newcommand{\ze}{\zeta}

\newcommand{\ga}{\gamma}
\newcommand{\gap}{\gamma_+}
\newcommand{\gam}{\gamma_-}

\newcommand{\Ga}{\Gamma}

\newcommand{\bfo}{{\bf 1}}

\newcommand{\cp}{{c_+}}
\newcommand{\cm}{{c_-}}

\newcommand{\Cp}{{C_+}}
\newcommand{\Cm}{{C_-}}
\newcommand{\supp}{{\mathrm{supp}}}

\begin{document}

\title[Expectations of functions of a stable L\'evy process and its extremum]
{Efficient evaluation  of expectations of functions of a stable L\'evy process and its extremum}

\author[
Svetlana Boyarchenko and
Sergei Levendorski\u{i}]
{
Svetlana Boyarchenko and
Sergei Levendorski\u{i}}

\begin{abstract}
Integral representations for expectations of functions of a stable L\'evy process $X$ and its supremum $\bar X$ are derived.  
As examples, cumulative probability distribution functions (cpdf) of $X_T, \barX_T$, the joint cpdf of $X_T$ and $\barX_T$, and the expectation of $(\be X_T-\barX_T)_+$, $\be>1$, are considered, and efficient numerical    procedures for cpdfs are developed. The most efficient numerical methods use the conformal acceleration technique and simplified trapezoid rule. 

\end{abstract}

\thanks{
\emph{S.B.:} Department of Economics, The
University of Texas at Austin, 2225 Speedway Stop C3100, Austin,
TX 78712--0301, {\tt sboyarch@utexas.edu} \\
\emph{S.L.:}
Calico Science Consulting. Austin, TX.
 Email address: {\tt
levendorskii@gmail.com}}

\maketitle

\noindent
{\sc Key words:} stable L\'evy processes, extrema of a stable L\'evy process, fractional partial differential equations, Fourier transform,  Gaver-Wynn Rho algorithm, sinh-acceleration, conformal acceleration technique

\noindent
{\sc MSC2020 codes:} 26A33, 33C48, 35R11, 65M70, 65T99, 65G70, 60-08, 60G52, 42A38, 42B10, 44A10,65R10,65G51,91G20,91G60

\tableofcontents

\section{Introduction}\label{s:intro}
Stable L\'evy processes and densities appear in various fields of natural sciences, engineering and finance.
 See, e.g., 
\cite{Chandrasekhar,Chavanis09,Mand1,Mand2,Zolotarev,SamorodnitskyTaqqu94,Weron96,MitRachev,Nolan97,Nolan98,Nolan03,NikiasShao95,SignalProc10,Saenko16,AmentONeil18} and the bibliographies therein. Efficient calculation of (cumulative) probability distribution functions  (cpdf and pdf) and expectations in stable L\'evy models are non-trivial even in the one-dimensional case (1D case).  The slow decay of the transition kernels at infinity and very large derivatives of the kernels near the peak make accurate calculations in the state space difficult. If the Fourier transform technique is used, standard numerical quadratures face serious difficulties because of  the oscillation and slow decay of the integrands at infinity. In particular,
standard tools such as the fast Fourier transform produce very large errors unless the index $\al$ of the stable L\'evy process is  close to 2, hence, the process is close to the Brownian motion (BM).  
The methods in op.cit. are very difficult to generalize to subordinated stable L\'evy processes and mixtures of stable distributions, 
which arise in applications to signal processing  \cite{SignalProc10}. See \cite{ConfAccelerationStable} for details.
 Generalizations to distributions of extrema of a L\'evy processes
and joint distributions, which we consider
in the present paper, is essentially impossible. It is necessary to evaluate
  double and triple integrals, hence, the detailed analysis of the integrands used in the 1D case does not seem feasible. 
  
In the present paper, as in \cite{ConfAccelerationStable} and  \cite{EfficientLevyExtremum}, we use the conformal deformation technique. 
The conformal deformations in  \cite{EfficientLevyExtremum} are applicable to L\'evy processes with the characteristic exponents admitting analytic continuation to a strip; in  \cite{ConfAccelerationStable},  variations of the conformal deformation technique suitable  for evaluation pdf and cpdf of stable
L\'evy distribution on $\bR$ are developed.  The new elements of the present
paper are efficient methods for evaluation of the Wiener-Hopf factors and 2D-3D integrals in the formula for
expectations  $V(f,T;x_1,x_2)=\bE[f(x_1+X_T, \max\{x_2, x_1+\barX_T\})]$ of functions of a stable L\'evy process  $X$  and its  
supremum process $\barX$ both starting at zero. The generalization to the case when the infimum process replaces supremum process is by symmetry. The formula is derived in \cite{EfficientLevyExtremum} for any L\'evy process; numerical realizations have to be modified for the case of stable L\'evy processes.
In the nutshell, the main idea is to deform the contours  of integration using  appropriate conformal maps, which must be in a certain agreement. The corresponding changes of variables and simplified trapezoid rule allow us to satisfy a small error tolerance, at a small CPU cost.   The efficiency of the simplified
trapezoid rule  stems from the fact that the discretization error of the  infinite trapezoid rule is  $O(\exp[-2\pi d/\ze])$, where $d$ is the half-width of the strip of analyticity around the line of integration and $\ze$ is the step size. As in \cite{ConfAccelerationStable}, certain
preliminary regularizations of the integrals are needed. In many cases, the method is faster and more accurate than the saddle point method and methods based on the reduction to an appropriate cut in the complex plane (see, e.g., \cite{Fedoryuk,NG-MBS,IAC}).
The schemes that we design use  the general analytical properties of the characteristic exponent and can be easily modified to more involved cases of mixtures of stable distributions and subordinated
stable L\'evy processes, and  to solution of several types of boundary problems for fractional 
differential equations; generalizations to higher dimensions are possible. 

Special cases which we consider in more detail in the paper, are: (1) $f$ depends on the first argument only; (2) $f(x_1,x_2)=\bfo_{(\infty,a]}(x_2)$;
(3) $f(x_1, x_2)=\bfo_{(-\infty, a_1]}(x_1)\bfo_{(-\infty, a_2]}(x_2)$, where $x_1\le x_2\le a_2, a_1\le a_2$; (4)
$f(x_1,x_2)=(\be x_1-x_2)_+e^{-\la x_2}$, where $\be>1$ and $\la>0$. 
Case (1), which we consider in Section \ref{s:1D}, is a straightforward generalization of the case 
of the cpdf of a stable L\'evy process considered in \cite{ConfAccelerationStable}. Numerical schemes outlined for this case are used as  blocks in
more complicated situations. 
First, as in \cite{ConfAccelerationStable}, we calculate
the expectation $\bE[f(x+X_T)]$ using the Fourier transform technique. 
We reduce the integral in the Fourier inversion formula to an integral over
over a half-axis, rotate the half-axis to $e^{i\om}\bR_+$, where the choice of $\om\in (-\pi/2,\pi/2)$ is determined by  the domains of analyticity and oscillation of the characteristic function of the process and Fourier transform $\hf$. Then we make the change of variables $\xi=e^{i\om+y}$ and apply the simplified trapezoid rule. The exponential rate of decay of the error of the infinite trapezoid rule as a function of $1/\ze$ allows one to easily satisfy a small error tolerance unless $d$, the half-width of the strip of analyticity, is very small. The $d$ is very small in the case of strongly 
 asymmetric L\'evy processes of index $\al\in (0,2), \al\neq 1$; for asymmetric L\'evy processes of index 1, the rotation is impossible for
  $x$ in a certain semi-infinite interval unless high precision arithmetic is used. To alleviate these difficulties, in \cite{ConfAccelerationStable}, we used two additional families of conformal deformations; the CPU time significantly increased. In the present paper, we calculate of
$\bE[f(x+X_T)]$ using the Laplace transform w.r.t. $T$ and Fourier transform w.r.t. $x$. The resulting 2D integral is calculated using either the Gaver-Wynn Rho algorithm (GWR algorithm) for the Laplace inversion or the conformal deformationof the line of integration in the Bromwich integral  used in \cite{Contrarian,EfficientLevyExtremum}.  The deformation of the same type, namely, the sinh-acceleration, was introduced in \cite{Sinh} in the context of evaluation of special functions and used in \cite{SINHregular} to price European options in wide classes of L\'evy and affine models. The inner integral (the inverse Fourier transform) can be calculated using the rotation of the half-axis of integration and the exponential change of variable $\xi=e^{i\om+y}$. The half-width of the strip of analyticity $d=\pi/8$ in the new coordinate $y$ is not small, hence, the number of terms in the simplified trapezoid rule for the inner integral is small,
which compensates for the drawback of calculations for several dozen of points in the Laplace inversion formula. 


Case (2) is the probability distribution of the supremum process. We calculate the Laplace transform w.r.t. $T$ using
the Wiener-Hopf factorization technique, and evaluate the Bromwich integral. The result is a 2D integral, the integrand 
being expressed in terms of the Wiener-Hopf factors. An efficient numerical procedure uses the conformal deformations
in the formula for the Wiener-Hopf factor and in the 2D integral; all deformations must be in a certain agreement.
A simpler but less accurate version uses GWR algorithm to evaluate the Bromwich integral.
The elements of the Wiener-Hopf factorization technique used in the paper are in Section \ref{ss:WHF_gen}.
In particular, we introduce a decomposition of the Wiener-Hopf factors, which is necessary for  efficient realizations of the
general formula for $V(f,T;x_1,x_2)$. The formula is derived in \cite{EfficientLevyExtremum}, for any L\'evy process, and
an efficient numerical realization is developed for L\'evy processes with exponentially decaying tails of the L\'evy density.
The decomposition allows us to repeat the constructions in \cite{EfficientLevyExtremum} almost verbatim.
The main difference is that in \cite{EfficientLevyExtremum}, the sinh-changes of variables on $\bR$, of the form
$\xi=i\om_1+b\sinh(i\om+y)$,
where $\om_1\in \bR$, $b>0$ and $\om\in (0,\pi/2)$ are used; in the present paper, 
the exponential changes of variables on rays $e^{i\om}\bR_+$ are used: $\xi=e^{i\om +y}$. 
As it was demonstrated in  \cite{ConfAccelerationStable}, typically, the sinh-change of variables and exponential change of variables are the most efficient ones but, in some cases, less efficient families are more efficient or even indispensable. 

Case (3) is the joint cpdf of the stable L\'evy process and its extremum. If GWR algorithm is used, then, for each $q>0$ in GWR, we need
to evaluate 2D integral; if the sinh-acceleration is applied to the Browmwich integral, then 3D integral. The latter version allows us to
achieve much higher precision than the former but it is possible only if the index $\al>1$ or $\al=1$ and the jump part is symmetric or $\al\in (0,1)$ and $\mu=0$. Note that in Case (3), $V(f,T;x_1,x_2)$ is independent of $x_2$.
In Case (4), $V(f,T;x_1,x_2)$ depends on both $x_1$ and $x_2$. Furthermore, Case (4) is used to illustrate how the general theorem formulated and proved for bounded measurable functions can be used
when $f$ is unbounded: introduce the dampening factor, derive the integral representation and prove that it is possible to pass to the limit $\la\to 0$. In Case (4), the limit is finite iff $\al>1$. 

The integral representations are derived in Section \ref{int_repr}. In Section \ref{numer_real}, efficient numerical realizations are 
discussed, the detailed algorithms for Case (3) are formulated, and  numerical examples are  produced.
Section \ref{concl} concludes; technical details and tables with numerical examples are relegated to appendices \ref{s:techn} and
\ref{s:tables}, respectively.

\section{Evaluation of expectations $E[f(x+X_T)]$. Conformal deformations technique}\label{s:1D}

\subsection{Stable L\'evy processes: main notation}\label{ss:stable_notation}
Let $X$ be a one-dimensional stable L\'evy process on the filtered probability space $(\Om, \cF, \{\cF_t\}_{t\ge 0}, \bP)$
satisfying the usual conditions, and let $\bE$ be the expectation operator under $\bP$.
We use the same parametrization of stable L\'evy processes on $\bR$ and representations of the characteristic exponent as in \cite{ConfAccelerationStable}. 
Let $c_\pm\ge 0$, $c_++c_->0$, $\al\in (0,2)$, and let $X$ be a stable L\'evy process of index $\al$, with the L\'evy density 
\[
F(dx)=c_+ x^{-\al-1}\bfo_{(0,+\infty)}(x)dx+c_-|x|^{-\al-1}\bfo_{(-\infty,0)}(x)dx.
\]
If $\al\neq 1$,  the characteristic exponent is of the form
\bbe\label{stdr}
\psi_{st}(\xi)=-i\mu\xi+\psi^0_{st}(\xi),
\ee
where $\mu\in\bR$,
\beqa\label{SLnuneq10}
\psi^0_{st}(\xi)&=&\Cp |\xi|^\al\bfo_{(0,+\infty)}(\xi)+ \Cm|\xi|^\al\bfo_{(-\infty,0)}(\xi),\ \xi\in \bR,
\\\label{Cp}
\Cp&=&\Cp(\al,\cp,\cm)=-\cp\Ga(-\al)e^{-i\pi\al/2}-\cm\Ga(-\al)e^{i\pi\al/2},
\eqa
and $\Cm=\overline{\Cp}$. 
If $\al=1$, then the formula for $\psi^0_{st}$ is more involved:
  \bbe\label{SLnu1}
\psi^0_{st}(\xi)=\sg|\xi|(1+i(2\be/\pi)\,\mathrm{sign}\,\xi\ln|\xi|),\ \xi\in \bR,
\ee
where $\sg=(\cp+\cm)\pi/2$, $\be=(\cp-\cm)/(\cp+\cm)$. This is a version of Zolotarev's  parametrizations \cite{Zolotarev}
for stable processes of index 1.  In the symmetric case $c=c_+=c_-$, \eq{SLnu1} reduces to 
  \eq{SLnuneq10} with $\Cp=\Cm=c\pi$ and $\al=1$. 
\begin{prop}\label{prop:anal_psi_st} 
\begin{enumerate}[(a)]
\item
$\psi^0_{st}(\xi)$ admits analytic continuation from $(0,+\infty)$ to the right half-plane (and to an appropriate Riemann surface).
\item
$\psi^0_{st}(\xi)$ admits analytic continuation from $(-\infty,0)$ to the left half-plane (and to an appropriate Riemann surface).
\end{enumerate}
\end{prop}
\begin{proof} It suffices to note that  $\xi^\al=\exp[\al\ln \xi]$ for $\xi\in\bC\setminus (-\infty,0]$.
\end{proof}
We use  the notation
 $\cC_{\gam,\gap}=\{e^{i\varphi}\rho\ |\ \rho> 0, \varphi\in (\gam,\gap)\cup (\pi-\gap,\pi-\gam)\}$, 
 $\cC_{\ga}=\{e^{i\varphi}\rho\ |\ \rho> 0, \varphi\in (-\ga,\ga)\}$. The proof of the following lemma is by inspection. 
\begin{lem}\label{lem:asymp_psi_st}  For any $\ga\in (0,\pi/2)$, as $(\cC_\ga\ni)\xi\to \infty$,
\begin{enumerate}[(a)]
\item
if $\al\in (1,2)$ or $\al\in(0,1)$ and $\mu=0$,
\bbe\label{aspsimain}
\psi_{st}(\xi)=\Cp \xi^\al 
(1+|\mu|O(|\xi|^{1-\al})),
\ee
and
\bbe\label{aspsimainRe}
\Re\psi_{st}(\xi)=|\Cp|\cos(\varphi_0+\varphi) |\xi|^\al 
(1+|\mu|O(|\xi|^{1-\al})),
\ee
where $\varphi_0=\mathrm{arg}\,\Cp$ and $\varphi=\mathrm{arg}\,\xi$;
\item
if $\al\in (0,1)$ and $\mu\neq 0$,
\bbe\label{aspsi01}
\psi_{st}(\xi)=-i\mu \xi(1+O(|\xi|^{\al-1})),
\ee
and
\bbe\label{aspsi01Re}
\Re\psi_{st}(\xi)=\mu \sin(\varphi)|\xi|(1+O(|\xi|^{\al-1}));
\ee
\item
if $\al=1$ and $c_+=c_-=c$, then 
\bbe\label{aspsi1eq}
\psi_{st}(\xi)=(c\pi-i\mu) \xi,
\ee
and
\bbe\label{aspsi1eqRe}
\Re\psi_{st}(\xi)=\sqrt{(c\pi)^2+\mu^2}\cos(-\arctan(\mu/(c\pi))+\varphi) |\xi|;
\ee
\item
if $\al=1$ and $c_+\neq c_-$, then 
\bbe\label{aspsi1neq}
\psi_{st}(\xi)=((c_++c_-)\pi/2-i\mu)\xi+i(c_+-c_-)\xi\ln\xi,
\ee
and
\bbe\label{aspsi1neqRe}
\Re\psi_{st}(\xi)=(c_--c_+)\sin(\varphi)|\xi|\ln|\xi|+O(|\xi|).
\ee
Furthermore, if $c_+>c_-$, then there exists $R>0$ s.t. for   $\xi\in \{\xi \in \cC_{-\pi/2,0}\ | \ |\xi|\ge R\}$,
\bbe\label{aspsi1neqRe1p}
\Re\psi_{st}(\xi)\ge \sqrt{\sg^2+\mu^2}\cos(\varphi-\arctan(\mu/\sg)|\xi|,
\ee
and if $c_+<c_-$, then \eq{aspsi1neqRe1p} holds for   $\xi\in \{\xi \in\cC_{0,\pi/2}\ |\ |\xi|\ge R\}$.
\end{enumerate}
\end{lem}

\begin{lem}\label{SLroots}
For 
any $\xi\in \bC\setminus i\bR$ and $q>0$, $q+\psi_{st}(\xi)\neq 0$.
\end{lem}
\begin{proof}
This property is proved in \cite{EfficientAmenable} for a wide class of L\'evy processes
called {\em Stieltjes-L\'evy processes} (SL-processes). It is also proved that stable L\'evy processes are SL-processes.
\end{proof}

If $\al\ge 1$ or $\al\in (0,1)$ and $\mu=0$, set $\bar\al=\al$. If $\al\in (0,1)$ and $\mu\neq 0$, set $\bar\al=1$.

The following lemma is immediate from Lemma \ref{lem:asymp_psi_st}.
\begin{lem}\label{lem:psi_st_bounds}
\begin{enumerate}[(a)]
\item
For any $\ga\in (0,\pi/2)$ and any $q>0$, there exists $C>0$ such that
\bbe\label{bound_psi_inv}
q/|q+\psi_{st}(\xi)|\le C(1+|\xi|)^{-\bar\al},\ \xi\in \cC_\ga.
\ee
\item
There exist $-\pi/2<\gam\le 0\le \gap<\pi/2$ and $c, R>0$ such that $\gam<\gap$ and
\bbe\label{lim_bound_Re_psi}
\Re\psi^0_{st}(\xi)\ge c|\xi|^{\al},\ \xi\in \cC_{\gam,\gap}, |\xi|\ge R.
\ee
\item
\begin{enumerate}[(i)]
\item
If $\al\in (0,1)\cup (1,2)$ or $\al=1$ and $c_+=c_-$, then $\gam<0<\gap$.
\item
If $\al=1$ and $c_+>c_-$ (resp., $c_+<c_-$), then $\gap=0$ (resp., $\gam=0$).
\end{enumerate}
\end{enumerate}
\end{lem}
\begin{rem}\label{rem:Re vs upper bound}{\rm
The line of integration in the Fourier inversion formula in the next subsection can be deformed into a cone where \eq{lim_bound_Re_psi}
holds, and the lines of integration in the formulas in the following sections - into a cone where
\eq{bound_psi_inv} holds. This fact implies that if GWR algorithm is used to evaluate the Bromwich integral, the restrictions on the type
of deformations in the formulas for the Laplace transform $\tV(f;q,x_1,x_2)$ are milder than when the Fourier inversion is used
to evaluate $\bE[f(x+X_T)]$. One can use the Laplace-Fourier inversion to evaluate $\bE[f(x+X_T)]$ which means that 
we  replace a 1D integral
with 2D. We will analyze in a separate publication the advantages of this  approach in applications to evaluation of pdf and cpdf of strongly asymmetric stable L\'evy processes, of index $\al=1$ or close to 1 especially.
}
\end{rem}

\subsection{Evaluation of $V(T,x)=\bE[f(x+X_T)]$}\label{ss:1D} Representing $f:\bR\to \bC$ as a linear combination of positive functions supported on either  $(-\infty,0]$ or  $[0,+\infty)$, we may restrict ourselves to non-negative $f$ supported on $[0,+\infty)$ only.

\vskip0.1cm
\noindent
{\sc Assumption $f$.} The following conditions are satisfied:
\begin{enumerate}[(I)]
\item
$\supp\, f\in [0,+\infty)$, and,
 for any $\la>0$, $f(\la,\cdot):=f(\cdot)e^{-\la\cdot}\in L_1(\bR;
 \bR_+)$;
 \item  there exist $a\in \bR$ and $\ga_0\in (\pi/2,\pi)$ such that the Fourier transform $\hf$ is of the form $\hf(\xi)=e^{-ia\xi}\hf_0(\xi)$,
where $\hf_0(\xi)$ admits analytic continuation to  $-i\cC_{\ga_{0}}$;

\item there exist  $C>0$, $\al_0<\al$ and $\al_\infty\in \bR$, such that 
\bbe\label{boundfD1}
|\hf_0(\xi)|\le C(|\xi|^{-\al_0-1}\bfo_{(0,1]}(|\xi|)+|\xi|^{\al_\infty-1}\bfo_{[1,+\infty)}(|\xi|)),\ \xi\in -i\cC_{\ga_{0}}.
\ee 
\end{enumerate}
Set $V(\la;T,x)=\bE[f(x+X_T)e^{-\la (x+X_T)}]$.
We have $V(T,x)=\lim_{\la\downarrow 0}V(\la;T,x)$. Let $Y$ be the BM of variance $2|\Cp|$, with drift $\mu$, independent of $X$. Since $\hf_0(\la;\xi)$ 
admits analytic continuation to $\la-i\cC_{\ga_{0}}$ and is polynomially bounded at infinity, we can use the Fourier transform technique.
Set $x'=x+\mu T-a$. We have 
\beqast
V(\la;T,x)&=&\frac{1}{2\pi}\int_\bR e^{ix'\xi-T\psi^0_{st}(\xi)}\hf_0(\la;\xi)d\xi\\
&=&\frac{1}{2\pi}\int_\bR e^{ix'\xi}\left(e^{-T\psi^0_{st}(\xi)}-e^{-T|\Cp|\xi^2}\right)\hf_0(\la;\xi)d\xi +\bE[f(\la;x+Y_T)].
\eqast
As $\xi\to 0$, $e^{-T\psi_{st}(\xi)}-e^{-T|\Cp|\xi^2}=O(|\xi|^\al)$, and $\hf(\la,\xi)=O(|\xi|^{-\al_0-1})$, uniformly in $\la$. Hence, we can pass to the limit $\la\to 0$, and obtain $V(T,x)=V_1(T,x)+\bE[f(x+Y_T)]$, where
\bbe\label{eq:VT}
V_1(T,x)=\frac{1}{2\pi}\int_\bR e^{ix'\xi}\left(e^{-T\psi^0_{st}(\xi)}-e^{-T|\Cp|\xi^2}\right)\hf_0(\xi)d\xi.
\ee
For the BM model, efficient numerical procedures are well-known (see, e.g., \cite{SINHregular} and the bibliography therein),
hence, it suffices to evaluate  $V_1(T,x)$. 
Since $\overline{e^{ix'\xi}}=e^{-ix'\bar\xi}$, $\overline{\psi^0_{st}(\xi)}=\psi^0_{st}(-\bar\xi)$ and $\overline{\hf_0(\xi)}=\hf_0(-\bar\xi)$,
we have
\bbe\label{V1Tx}
V_1(T,x)=\frac{1}{\pi}\Re\int_0^{+\infty} e^{ix'\xi}\left(e^{-T\psi^0_{st}(\xi)}-e^{-T|\Cp|\xi^2}\right)\hf_0(\xi)d\xi.
\ee
Assume that either $\al\neq 1$ or $\al=1$ and $c_+=c_-$. Then there exist $\gam<0<\gap$ such that \eq{lim_bound_Re_psi} holds.
We redefine $\gap:=\min\{\gap,\ga_0-\pi/2\}$, and, if $x'>0$,  take $\om\in (0,\gap/2)$, deform the line of integration into $e^{i\om}\bR_+$, and change the variable $\xi=e^{i\om+y}$:
\bbe\label{V1Tx2}
V_1(T,x)=\frac{1}{\pi}\Re\int_{-\infty}^{+\infty} e^{ix'\xi(y)}\left(e^{-T\psi^0_{st}(\xi(y))}-e^{-T|\Cp|\xi(y)^2}\right)\hf_0(\xi(y))e^{i\om+y}dy.
\ee
If $x'<0$ (resp., $x'=0$), we use \eq{V1Tx2} with $\om\in (\gam/2,0)$ (resp., $\om=(\gap+\gam)/2$).  We evaluate
the integral on the RHS of \eq{V1Tx2} using the infinite trapezoid rule. 

The integrand on the RHS of \eq{V1Tx2}, denote it  $g$, is analytic in the strip
$S_{(-d,d)}:=\{\xi\ | \Im\xi\in (-d,d)\}$, for any $d\in (0,\gap/2)$ (if $x'>0$, $d\in (0,-\gam/2)$; if $x'=0$, 
$d\in (0, (\gap-\gam)/2)$) and decays at infinity sufficiently fast so that
$\lim_{A\to \pm\infty}\int_{-d}^d |g(i a+A)|da=0,$
and 
\bbe\label{Hnorm}
H(g,d):=\|g\|_{H^1(S_{(-d,d)})}:=\lim_{a\downarrow -d}\int_\bR|g(i a+ y)|dy+\lim_{a\uparrow d}\int_\bR|g(i a+y)|dy<\infty
\ee
is finite. We write $g\in H^1(S_{(-d,d)})$. The integral
$I=\int_\bR g(\xi)d\xi$
can be evaluated using the infinite trapezoid rule
\bbe\label{inftrap}
I\approx \ze\sum_{j\in \bZ} g(j\ze),
\ee
where $\ze>0$. 
The following key lemma is proved in \cite{stenger-book} using the heavy machinery of sinc-functions. A simple proof can be found in
\cite{paraHeston}.
\begin{lem}[\cite{stenger-book}, Thm.3.2.1] Let $g\in H^1(S_{(-d,d)})$. 
The error of the infinite trapezoid rule \eq{inftrap} admits an upper bound 
\bbe\label{Err_inf_trap}
{\rm Err}_{\rm disc}\le H(g,d)\frac{\exp[-2\pi d/\ze]}{1-\exp[-2\pi d/\ze]}.
\ee
\end{lem}
Once
an approximately bound for $H(g,d)$ is derived, it becomes possible to satisfy the desired error tolerance
with a good accuracy.  The infinite sum is truncated
\bbe\label{inftrap_trunc}
I\approx \ze\sum_{j=-N_-}^{N_+} g(j\ze).
\ee
As $j\to +\infty$, $g(j\ze)$ exhibits a super-exponential decay, hence, a fairly small $N_+$ can be chosen to satisfy even a very small error tolerance $\eps$. In a neighborhood of  $-\infty$, $g(j\ze)$ admits a bound via $C\exp[(\al-\al_0)j\ze]$, where an approximate upper bound for $C$ can be easily derived. Hence, we choose $N_-=\mathrm{ceil}\,(\ln(C/\eps)/(\al-\al_0))$. If $\hf_0$ admits an asymptotic expansion
as $\xi\to 0$, the asymptotic expansions of $e^{ix'\xi}, e^{-T\psi^0_{st}(\xi)}$ and $e^{-T|\Cp|\xi^2}$ as $\xi\to 0$ can be used
to calculate an asymptotic expansion 
\[
g(j\ze)=\sum_{k=1}^{M-1} c_k e^{\be_k j\ze}+ O(e^{\be_M j\ze}),
\]
where $c_k$ and $\be_k>0$ depend on $\hf_0, \al, c_\pm$ and $T$.  After that we calculate explicitly the sums
$\sum_{j=-\infty}^{-N_--1}e^{\be_k j\ze}$, and choose significantly smaller $N_-=\mathrm{ceil}\,(\ln(C/\eps)/\be_M)$.
See \cite{ConfAccelerationStable} for details. 
\begin{rem}\label{rem:other families}{\rm The complexity of the scheme based on the exponential change of variables admits a bound via
 $C(\ln(1/\eps))^2$. As we demonstrated in \cite{ConfAccelerationStable}, if either $\gap$ or $-\gam$ are very small,
then the half-width of the strip of analyticity in the $y$-coordinate is very small, $\ze$ is very small, and the constant $C$ is extremely
large. In the result, even for a fairly small error tolerance, e.g., E-15, it may be
advantageous to use a polynomial change of variables. The complexity of the corresponding scheme is of the order of
$C\eps^{-m}\ln(1/\eps)$, where $m>0$. Finally, if $\al=1$ and $c_+\neq c_-$, then, depending on the sign of $x'(c_+-c_-)$,
neither the exponential change of variables nor polynomial one can be used, and only a sub-polynomial increase of the rate of the decay of
the integrand can be achieved. The number of terms and CPU time are much larger.

}
\end{rem}

\subsection{An alternative evaluation of $V(T,x)=\bE[f(x+X_T)]$ and GWR algorithm}\label{ss:1D_alt} The trick in this subsection allows us to use the exponential change of variables in all cases, at the cost of an additional integration (Laplace inversion).
Let $f$  satisfy the same conditions as above, with an additional restriction on $\al_\infty$ in \eq{boundfD1}. Namely,
we assume that $\al_\infty<\bar\al$, where 
$\bar\al$ is as in Lemma \ref{lem:psi_st_bounds}. Somewhat counterintuitively, we represent the 1D integral \eq{eq:VT}
as a 2D integral
\bbe\label{eq:V1TBr}
V(T,x)=\frac{1}{2\pi i}\int_{\Re q=\sg}dq\, e^{qT}\frac{1}{2\pi}\int_\bR e^{ix'\xi}\left(\frac{1}{q+\psi^0_{st}(\xi)}-\frac{1}{q+|\Cp|\xi^2}\right)\hf_0(\xi)d\xi.
\ee
If the outer integral (Bromwich integral) is evaluated using GWR algorithm, it suffices to evaluate the inner integral for a moderate number of values of $q>0$. Similarly to \eq{V1Tx}, for each $q$, we simplify the inner integral 
\bbe\label{tVq1D}
\tV(q,x) =\frac{1}{\pi}\Re \int_0^{+\infty}e^{ix'\xi}\left(\frac{1}{q+\psi^0_{st}(\xi)}-\frac{1}{q+|\Cp|\xi^2}\right)\hf_0(\xi)d\xi.
\ee
On the strength \eq{bound_psi_inv}, we can rotate the ray of integration to $e^{i\om}\bR_+$ using any $\om\in (0,\pi/4)$ (resp., $\om\in (-\pi/4,0)$;
$\om\in (-\pi/4,\pi/4)$) if $x'>0$ (resp., $x'<0$; $x'=0$), and the choice $\om=\pi/8$ (resp., $\om=-\pi/8$; $\om=0$) is optimal.
After the change of variables $\xi=e^{i\om+y}$, we apply the simplified trapezoid rule. The half-width of the strip of analyticity
of the integrand in the $y$-coordinate, which can be used to derive the recommendation for the choice of $\ze$, is
an arbitrary $d\in (0,\pi/8)$, if $x'\neq 0$, and $d\in (0,\pi/4)$, if $x'=0$. Thus, in all cases, we can use $\ze$ of the order of
$(\pi^2/4)/\ln(1/\eps)$, whereas in Section \ref{ss:1D}, if either $\gap$ or $-\gam$ is very small, $\ze$ must be very small as well.
Thus,   the number of terms in the simplified trapezoid rule for the evaluation of \eq{tVq1D} is many times smaller than
the number of terms in Section \ref{ss:1D}. Although we need to evaluate $\tV(q,x)$ for all $q$'s used in the algorithm chosen for the Laplace inversion, the number of $q$'s needed to satisfy even a very small error tolerance is measured in dozens, and
the calculations can be easily parallelized. Set
\[
g(y)=e^{ix'\xi(y)}\left(\frac{1}{q+\psi^0_{st}(\xi(y))}-\frac{1}{q+|\Cp|\xi(y)^2}\right)\hf_0(\xi(y))e^{i\om+y}.
\]
The truncation parameters $N_\pm$ in the simplified trapezoid rule 
\bbe\label{tVqsimp}
\tV(q,x)\approx \frac{\ze}{\pi}\Re\sum_{j=-N_-}^{N_+}g(j\ze)
\ee
are chosen taking into account the exponential asymptotics of $g(y)$ as $y\to\pm \infty$. Contrary to Section \ref{ss:1D}, we can explicitly calculate several terms of of the truncated parts of the infinite sum in both neighborhoods of $j=-\infty$ and $j=+\infty$, and
choose  smaller $N_\pm$.

\subsection{ Sinh-acceleration in the Bromwich integral}\label{ss:1D_alt_Br} 
Instead of GWR-algorithm, which typically requires high precision arithmetic to achieve the precision of the order better than E-7-E-8,
we can use the sinh-acceleration in the Bromwich integral (the outer integral on the RHS of \eq{eq:V1TBr}) and the exponential acceleration in the inner integral. This is possible under an additional conditions on the parameters of the process.
\begin{lem}\label{q_xi_cone}
\begin{enumerate}[(a)]
\item
Let either $\al\in (1,2)$ or $\al=1$ and $c_+=c_-$ or $\al\in (0,1)$ and $\mu=0$. 

Then
there exist $\sg>0, \ga_0\in (0,\pi/2)$ and $\gam<0<\gap$ such that
\bbe\label{asym_1}
q+\psi_{st}(\xi)\not\in (-\infty,0],\ \forall\ q\in [\sg,+\infty)+(\cC_{\pi/2+\ga_0}\cup\{0\}),\  \xi\in \cC_{\gam,\gap}. 
\ee
\item
If $\al\neq 1$ or $\al=1$ and $c_+=c_-$, then there exist $\sg>0, \ga_0\in (0,\pi/2)$ and $\gam<0<\gap$ such that $\psi^0_{st}$ 
satisfies \eq{asym_1}.
\item
If $\al=1$ and $c_+\neq c_-$ or $\al\in (0,1)$ and $\mu\neq 0$, then a collection  of $\sg>0, \ga_0\in (\pi/2,\pi)$ and $\gam<0<\gap$ such that
\eq{asym_1} holds does not exist.

\end{enumerate}
\end{lem}
\begin{proof} (a) It follows from Lemma \ref{lem:asymp_psi_st} that for any $\varphi\in (-\pi/2,\pi/2)$,
\[
\psi_{st}(\rho e^{i\varphi})\sim C e^{i(\varphi_1+\al\varphi)}\rho^\al,\ \rho\to +\infty,
\]
where $C>0$ and $\varphi_1\in (-\pi/2,\pi/2)$. Hence, for any $\de>0$, there exist $\gam<0<\gap$ and $\sg>0$ such that
for all $\xi\in \cC_{\gam,\gap}$, $\mathrm{arg}\, \psi_{st}(\xi)\in \cC_{\pi/2-\de}$. Then \eq{asym_1} holds for any $\ga_0\in (0,\de)$.

(b) is immediate from (a).

(c) If $\al=1$ and $c_+\neq c_-$, then, for any $\om\in (-\pi/2,\pi/2)$, there exists $c(\om)\in \bR\setminus\{0\}$
such that
  $
  \psi_{sr}(\rho e^{i\om})\sim c(\om) e^{i(\pi/2+\om)}\rho \ln\rho$ as $ \rho\to+\infty.
  $
  If $\al\in (0,1)$ and $\mu\neq 0$, then  the  asymptotic formula holds without the factor $\ln\rho$.
  It follows that if $a>0$ is sufficiently large, then there exists $\rho>0$ s.t. either $\sg+ia+\psi_{sr}(\rho e^{i\om})\in (-\infty,0]$
  or $\sg-ia+\psi_{sr}(\rho e^{i\om})\in (-\infty,0]$.
   \end{proof} 
Assume that either $\al\neq 1$ or $\al=1$ and $c_+=c_-$. Then $\psi^0_{st}$ satisfies the condition in (a). 
 We find $\sg, \ga_0, -\pi/4<\gam<0<\gap<\pi/4$ such that \eq{asym_1} holds for $\psi^0_{st}$
and $\psi^0_{BM}(\xi^2)$. Then choose $\sg_\ell>\sg$,
$\om_\ell\in (0,\ga_0)$, next,  $b_\ell>0$ such that  $\sg_\ell-b_\ell\sin\om_\ell>0$, and then define 
 \bbe\label{eq:sinhLapl}
\chi_{L; \sg_\ell,b_\ell,\om_\ell}(y)=\sg_\ell +i b_\ell\sinh(i\om_\ell+y).
\ee
 Deform 
 the line of integration in the Bromwich integral to
$\cL^{(L)}=\chi_{L; \sg_\ell,b_\ell,\om_\ell}(R)$. For $q\in \cL^{(L)}$, we can calculate
$\tV(q,x)$ choosing $\om\in (\gam,\gap)$ (the choice depends on the sign of $x'$ as in the case $q>0$) and deforming 
the line of integration on the RHS of \eq{eq:V1TBr} into the angle $e^{i\om}\bR_+\cup e^{i(\pi-\om)}\bR_+$:
\[
V (T,x)=\frac{1}{2\pi i}\int_{\cL^{(L)}}dq\, e^{qT}\frac{1}{2\pi}\left(\int_{e^{i(\pi-\om)}\bR_+}+\int_{e^{i\om}\bR_+}\right)e^{ix'\xi}\left(\frac{1}{q+\psi^0_{st}(\xi)}-\frac{1}{q+|\Cp|\xi^2}\right)\hf_0(\xi)d\xi.
\]
We make the change of variables $\xi=e^{i(\pi-\om)+y'}$ and $\xi=e^{i\om+y'}$ in the inner integrals, and 
$q=\chi_{L; \sg_\ell,b_\ell,\om_\ell}(y)$ in the outer integral. Each of the resulting integrals is calculated using the simplified trapezoid rule.

\section{Wiener-Hopf factorization for L\'evy processes and main theorem}\label{ss:WHF_gen} 
The definitions and results in this Section are valid for any L\'evy process $X$ on $\bR$. Let  $\barX_t=\sup_{0\le s\le t}X_s$ and $\uX_t=\inf_{0\le s\le t}X_s$ be the supremum
and infimum processes (defined path-wise, a.s.), and $X_0=\barX_0=\uX_0=0$. 
Let $q>0$ and let $T_q$
be an exponentially distributed random variable of mean $1/q$, independent of $X$. Set $\phipq(\xi)=\bE[e^{i\xi\barX_{T_q}}]$,
$\phimq(\xi)=\bE[e^{i\xi\uX_{T_q}}]$, and define the (normalized) expected present value operators $\cEq$ and $\cE^\pm_q$ by
$\cEq u(x)=\bE[u(x+X_{T_q})]$, $\cEpq u(x)=\bE[u(x+\barX_{T_q})]$, $\cEmq u(x)=\bE[u(x+\uX_{T_q})]$. Evidently, 
$\cEq, \cE^\pm_q: L_\infty(\bR)\to L_\infty(\bR)$ are bounded operators and, for $\xi\in\bR$, 
\[
\cEq e^{ix\xi}=\frac{q}{q+\psi(\xi)}e^{ix\xi},\ \cE^\pm_q e^{ix\xi}=\phi^\pm_q(\xi)e^{ix\xi}.
\]
The probabilisitic form 
 \bbe\label{whf0}
 \frac{q}{q+\psi(\xi)}=\phipq(\xi)\phimq(\xi)
 \ee
 of the Wiener-Hopf factorization identity is a special case of the Wiener-Hopf factorization of functions of wider classes than
 the ones arising in probability. Likewise, the operator form
  \bbe\label{eq:operWHF}
 \cEq=\cEpq\cEmq=\cEmq\cEpq
 \ee
 is a special case of the Wiener-Hopf factorziation used in the general theory of boundary problems
 for pseudo-differential operators (pdo), where more general classes of functions and operators appear
(see, e.g., \cite{eskin}) but additional regularity conditions on functions are imposed. The probabilistic version \eq{eq:operWHF} was proved in   \cite{BLSIAM02,NG-MBS,barrier-RLPE,EPV,IDUU} under additional regularity conditions on the process, and in \cite{single}, for any L\'evy process.
Both forms \eq{whf0} and \eq{eq:operWHF} are immediate from the following lemma; however, \eq{whf0} was derived earlier. See \cite{sato} for the bibliography.
\begin{lem}\label{l:deep} (\cite[Lemma 2.1]{greenwood-pitman}, and
\cite[p.81]{RW})
Let $X$ and $T_q$ be as above.  Then
\begin{enumerate}[(a)]
\item the random variables $\barX_{T_q}$ and
$X_{T_q}-\barX_{T_q}$ are independent; and

\item the random variables $\uX_{T_q}$ and
$X_{T_q}-\barX_{T_q}$ are identical in law.
\end{enumerate}
\end{lem}
(By symmetry, the statements (a), (b) are valid with $\barX$ and $\uX$ interchanged).
 By definition, part (a) amounts to the statement that the
probability distribution of the $\bR^2$-valued random variable
$(\barX_{T_q}, X_{T_q}-\barX_{T_q})$ is equal to the product (in the sense
of ``product measure'') of the distribution of $\barX_{T_q}$ and the
distribution of $X_{T_q}-\barX_{T_q}$. Using this observation, in \cite{EfficientLevyExtremum}, we proved the following theorem.
In the formulation of the theorem, we use the notation
$U_+=\{(x\in \bR^2\ |\ x_1\le x_2\}$;  $f_+$ is the extension of $f:U\to \bC$ by 0 to a function on $\bR^2$, $I$ denotes the identity operator, and $\De$ is the diagonal map: $\De(x)=(x,x)$. 
\begin{thm}\label{thm:X_barX_exp}
 Let $X$ be a L\'evy process on $\bR$, $q>0$, and let $f:U_+\to \bR$ be a measurable and uniformly bounded  function
s.t.   $((\cEmq\otimes I)f)\circ \De:\bR\to\bR$ is measurable.
 Then 
 \begin{enumerate}
  \item
 for any $x_1\le x_2$,
 \beqa\label{tVq0}
 q\tV(f,q;x_1,x_2)&=&((\cEq\otimes I)f_+)(x_1,x_2)+(\cEpq w(f;q,\cdot, x_2))(x_1),
 \eqa
 where 
 \bbe\label{eq:wqVtq}
 w(f;q,y,x_2)=\bfo_{[x_2,+\infty)}(y)(((\cEmq\otimes I)f_+)(y,y)-((\cEmq\otimes I)f_+)(y,x_2));
 \ee
  \item
   the RHS' of \eq{tVq0} and \eq{eq:wqVtq} admit analytic continuation w.r.t. $q$ to the right half-plane.
   \end{enumerate} 
 \end{thm}
 If an efficient numerical procedure for $\tV(f,q;x_1,x_2)$ is developed, we can calculate the expectation $V(f;T,x_1,x_2)$ as the Bromwich integral
\bbe\label{tVBrom}
V(f;T;x_1,x_2)=\frac{1}{2\pi i}\int_{\Re q=\sg}e^{qT}\tV(f;q;x_1,x_2)\,dq,
\ee
where $\sg>0$ is arbitrary. 

 \begin{rem}\label{rem:LapltVq}{\rm 
   In \cite{EfficientLevyExtremum}, for the case of L\'evy processes with characteristic exponent admitting analytical continuation to a strip,
we proved Theorem \ref{thm:X_barX_exp} for exponentially increasing $f$. In the case of stable L\'evy processes,
the restrictions on the rate of increase of $f$ at infinity are more stringent. If $f$ is unbounded, we recommend to calculate the expectation of the exponentially damped $f$, apply Theorem \ref{thm:X_barX_exp}, and calculate the limit of the resulting analytic expression as the dampening parameter tends to 0. See and
Example \ref{ex:exchange_stable}.}
\end{rem}
Note the following special cases of \eq{tVq0}.
\begin{enumerate}[I.]
\item
European option on the underlying. Let $f(x_1,x_2)=g(x_1)$. Then
\bbe\label{perp_opt_X}
\tV(f,q,x_1,+\infty)=q^{-1}(\cEq g)(x_1).
\ee
We considered this case in Section \ref{s:1D}.
\item
European option on the supremum. Let $f(x_1,x_2)=g(x_2)$. Then, for $x_1\le x_2$,
\bbe\label{perp_opt_barX}
\tV(f,q,x_1,x_2)=q^{-1}(\cEpq g)(x_2).
\ee
\item
No-touch option with the upper barrier $a_2$ and payoff $g(X_T)$ at maturity. Then $f(x_1,x_2)=g(x_1)\bfo_{(-\infty,a_2]}(x_2)$,
and  $q\tV(f,q,x_1,x_2)$ is given by the RHS of \eq{tVq0} or, equivalently, $\tV(f,q,x_1,x_2)$ is given by either of the two formulas
\beqa\label{perp_opt_X_barr2}
\tV(f,q,x_1,x_2)&=&q^{-1}(\cEq g)(x_1)-q^{-1}(\cEpq\bfo_{(a_2,+\infty)}\cEmq  g)(x_1),
\eqa
and
\beqa
\label{perp_opt_X_barr}
\tV(f,q,x_1,x_2)&=&q^{-1}(\cEpq\bfo_{(-\infty,a_2]}\cEmq  g)(x_1). 
\eqa
 Eq. \eq{perp_opt_X_barr} was derived
in \cite{KoBoL,NG-MBS,BLSIAM02,barrier-RLPE} for wide classes of L\'evy processes and generalized for all L\'evy processes in \cite{single}.
\end{enumerate}
 
We finish the section with the outline of a scheme of a numerical calculation of the first term on the RHS of \eq{tVq0}; the numerical tools for the calculations of the second term are analyzed in the rest of the main body of the paper.
 Assume that $\widehat{(f_+)}(\xi_1,x_2)$, the Fourier transform of $f_+(x_1,x_2)$ w.r.t. $x_1$, is a sum of products of oscillating exponents multiplied by functions
analytic in a cone $\cC_\ga\cup (-\cC_\ga), \ga\in (0,\pi/2),$ around $\bR\setminus 0$ and admitting sufficiently good bounds in neighborhoods of 0 and $\infty$. Explicitly,
\bbe\label{reprfplus}
\widehat{(f_+)}(\xi_1,x_2)=\sum_{k=1}^N e^{-ia_k\xi_1}\hg_k(\xi_1,x_2),
\ee
and, for all $x_2\in\bR$ and $\xi\in \cC_\ga\cup (-\cC_\ga)$,
\bbe\label{boundg_k}
|\hg_k(\xi_1,x_2)|\le C_k(x_2)(|\xi_1|^{\al_0-1}\bfo_{|\xi|\le 1}+|\xi_1|^{-1+\de}\bfo_{|\xi|\ge 1}),
\ee 
where  $\al_0>-\al$, $ \de<\al$, are independent of $x_2$, $k$ and $\xi$, and $C_k(x_2)$ depend on $k$ and $x_2$ only. 

We represent $((\cEq\otimes I)f_+)(x_1,x_2)$ as a sum according to \eq{reprfplus}, and evaluate each term
as in Sections \ref {ss:1D_alt}-\ref{ss:1D_alt_Br} ($x_2$ is regarded as a parameter).  Then we apply
GWR algorithm. If $\al\in (1,2)$ or $\al=1$ and $c_+=c_-$ or $\al\in(0,1)$ and $\mu=0$, then the sinh-acceleration can be applied.

\section{Wiener-Hopf factorization for stable L\'evy processes}\label{s:WHF-StLevy}

\subsection{Case of real $q$}\label{ss:form_WHF_real}
In \cite{EfficientAmenable}, we defined a class of Stieltjes-L\'evy processes (SL-processes), and proved that
stable L\'evy processes are SL-processes. We also proved that the characteristic exponent $\psi$ of an SL-process
admits analytic continuation to $\bC\setminus i\bR$ and
enjoys the following important property:
\bbe\label{eq:eqSL}
\forall\ q>0\ \mathrm{and}\ \xi\in \bC\setminus i\bR,\ q+\psi(\xi)\neq 0.
\ee
Using \eq{eq:eqSL} and Proposition \ref{prop:anal_psi_st}, we derive
\begin{thm}\label{tm:anal_psi_st} Let $q>0$. Then 
\begin{enumerate}[(a)]
\item
 $\phipq(\xi)$ admits analytic continuation w.r.t. $\xi$ from the upper half-plane to the complex plane with the cut $i(-\infty,0]$ (and an appropriate Riemann surface)
 by 
 \bbe\label{anal_cont_phipq}
 \phipq(\xi)=q/((q+\psi_{st}(\xi))\phimq(\xi));
 \ee
 \item 
$\phimq(\xi)$ admits analytic continuation w.r.t. $\xi$ from the lower half-plane to the complex plane with the cut $i[0,+\infty)$ (and an appropriate Riemann surface)
 by 
  \bbe\label{anal_cont_phimq}
  \phimq(\xi)=q/((q+\psi_{st}(\xi))\phipq(\xi)).
  \ee
 \end{enumerate}
 \end{thm}
For $\omp\in (0,\pi/2)$ and $\omm\in (-\pi/2,0)$, define the contours $\cL^+_{\omp}=e^{i(\pi-\omp)}\bR_+\cup e^{i\omp}\bR_+$ and 
$\cL^-_{\omm}=e^{i(-\pi-\omm)}\bR_+\cup e^{i\omm}\bR_+$.
 The direction on each contour is from the left to the right. For $q>0$ and $\xi\in \bR\setminus\{0\}$, define
\beqa\label{phip1}
\phi^{+,'}_q(\xi)&=&\exp\left[\frac{1}{2\pi i}\int_{\cL^-_{\omm}}\frac{\xi\ln(1+\psi_{st}(\eta)/q)}{\eta(\xi-\eta)}d\eta\right],\\
\label{phim1}
\phi^{-,'}_q(\xi)&=&\exp\left[-\frac{1}{2\pi i}\int_{\cL^+_{\omp}}\frac{\xi\ln(1+\psi_{st}(\eta)/q)}{\eta(\xi-\eta)}d\eta\right],
\eqa
and set $\phi^{\pm,'}_q(0)=1$. On the strength of Lemma \ref{SLroots}, 
 for any $\om\in(0,\pi/2)$, the integrands on the RHS' of \eq{phip1} and  \eq{phim1} are well-defined.
Since $\ln (1+\psi_{st}(\xi)/q)=O(|\eta|^\al)$ as $\eta\to 0$, and $=O(|\eta|^\eps)$ as $\eta\to \infty$, for any $\eps>0$,
the integrals are finite.
\begin{lem}\label{lem:stableWHF}
For any $q>0$ and $\xi\in \bR$, $\phi^\pm_q(\xi)=\phi^{\pm,'}_q(\xi)$. 
\end{lem}
\begin{proof} It suffices to consider $\xi\in \bR\setminus\{0\}$. If $\xi>0$, we note that both integrands are analytic in 
$\cC_{\omm,\omp}$, decay at infinity faster than $|\xi|^{-3/2}$ and continuous up to the boundary of $\cC_{\omm,\omp}$. 
Hence, the  residue theorem is applicable, and
we obtain 
\[
\frac{1}{2\pi i}\int_{\cL^-_{\omm}}\frac{\xi\ln(1+\psi_{st}(\eta)/q)}{\eta(\xi-\eta)}d\eta-\frac{1}{2\pi i}\int_{\cL^+_{\omp}}\frac{\xi\ln(1+\psi_{st}(\eta)/q)}{\eta(\xi-\eta)}d\eta=\ln\frac{q}{q+\psi_{st}(\xi)},
\]
therefore, \eq{whf0} holds with the product $\phi^{+,'}_q(\xi)\phi^{-,'}_q(\xi)$ on the RHS. It is known that the Wiener-Hopf factors $\phi^\pm_q(\xi)$
and their reciprocals are polynomially bounded as $\xi\to\infty$ in the corresponding half-plane.  The proof of Lemma \ref{asympWHF} below implies that  $\phi^{+,'}_q(\xi)$ and $\phi^{-,'}_q(\xi)$ are polynomially bounded, and a modification of same proof
can be used to prove that the reciprocals are polynomially bounded. The standard argument based on Moreira's theorem gives 
$\phi^\pm_q(\xi)=\phi^{\pm,'}_q(\xi)$.
\end{proof}

  Define 
$\varphi_0= \varphi_0(\al,\cp,\cm)=\mathrm{arg}\,\Cp(\al,\cp,\cm)$, and
\begin{enumerate}[i.]
\item
if $\al\in (1,2)$ or $\al\in (0,1)$ and $\mu=0$, set $\al^\pm=\al/2\pm \varphi_0/\pi$;
\item
if $\al\in (0,1)$ and $\mu>0$, set $\alp=1, \alm=0$;
\item
if $\al\in (0,1)$ and $\mu<0$, set $\alp=0, \alm=1$;
\item
if $\al=1$ and $c_+=c_-$, set $\alp=\alm=1/2$.
\end{enumerate} 
 The following lemma is proved  in Section \ref{proof:lem_asympWHF}; the asymptotic formulas are valid uniformly in $\xi\in \cC_\ga$, 
 for any $\ga\in (0,\pi/2)$.
 \begin{lem}\label{asympWHF}
  Let $q>0$ and $\ga\in (0,\pi/2)$. Then, 
 \begin{enumerate}[(1)]
 \item
 as $(\cC_\ga\ni)\xi\to 0$, 
 \bbe\label{asphim0}
 \phi^\pm(\xi)=1+O(|\xi|^{\min\{1,\al\}});
 \ee 
 \item
if either
$\al\in (1,2)$ and both $c_\pm>0$ or $\al\in(0,1)$ and $\mu=0$ or $\al=1$ and $c_+=c_-$, then both 
$\al_\pm>0$ and,
 as $(\cC_\ga\ni)\xi\to \infty$,
  \bbe\label{asphipmqmain}
\phi^\pm_q(\xi)=O(|\xi|^{-\al_\pm});
\ee
\item
\begin{enumerate}[(a)]
\item
if $\al\in (1,2)$ and $c_-=0$, then $\alp=\al, \alm=0$, and  $\exists\ \de>0$ s.t. as $(\cC_\ga\ni)\xi\to \infty$,
\beqa\label{asphipqp12}
\phipq(\xi)&=&O(|\xi|^{-\al}),\\\label{asphimqp12}
\phimq(\xi)&=&a^-_q + O(|\xi|^{-\de}),
\eqa
where
\bbe\label{amq12}
a^-_q=\exp\left[-\frac{1}{2\pi i}\int_{\cL^+_{\omp}}\frac{\ln(1-i\mu\eta/(q+\psi^0_{st}(\eta))}{\eta}d\eta\right],
\ee
and $\omp\in (0,\pi/2)$ is arbitrary;
\item
if $\al\in (1,2)$ and $c_+=0$, then $\alp=0, \alm=\al$, and $\exists\ \de>0$ s.t. as $(\cC_\ga\ni)\xi\to \infty$,

\beqa\label{asphipqm12}
\phipq(\xi)&=&a^+_q + O(|\xi|^{-\de}),\\\label{asphimqm12}
\phimq(\xi)&=&O(|\xi|^{-\al}),
\eqa
where
\bbe\label{apq12}
a^+_q=\exp\left[\frac{1}{2\pi i}\int_{\cL^-_{\omm}}\frac{\ln(1-i\mu\eta/(q+\psi^0_{st}(\eta))}{\eta}d\eta\right],
\ee
and $\omm\in (-\pi/2,0)$ is arbitrary;
\item
if $\al\in (0,1)$ and $\mu> 0$, then there exists $\de>0$ such that as $(\cC_\ga\ni)\xi\to \infty$,
\beqa\label{asphipqp01}
\phipq(\xi)&=&O(|\xi|^{-1}),\\\label{asphimqp01}
\phimq(\xi)&=&a^-_q + O(|\xi|^{-\de}),
\eqa
where
\bbe\label{amq}
a^-_q=\exp\left[-\frac{1}{2\pi i}\int_{\cL^+_{\omp}}\frac{\ln((1+\psi^0(\eta)/q)/(1-i\mu\eta/q))}{\eta}d\eta\right],
\ee
and $\omp\in (0,\pi/2)$ is arbitrary;

\item
if $\al\in (0,1)$ and $\mu< 0$, then then there exists $\de>0$ such that as $(\cC_\ga\ni)\xi\to \infty$,
\beqa\label{asphipqm01}
\phipq(\xi)&=&a^+_q+ O(|\xi|^{-\de}),\\\label{asphimqm01}
\phimq(\xi)&=&O(|\xi|^{-1}),
\eqa
where  
\bbe\label{apq}
a^+_q=\exp\left[\frac{1}{2\pi i}\int_{\cL^-_{\omm}}\frac{\ln((1+\psi^0(\eta)/q)/(1-i\mu\eta/q))}{\eta}d\eta\right],
\ee
and $\omm\in (-\pi/2,0)$ is arbitrary;

\item
if $\al=1$ and $c_+>c_-$, then, for any $\eps>0$, as $(\cC_\ga\ni)\xi\to \infty$,
\beqa\label{asphipq1p}
\phipq(\xi)&=&O(|\xi|^{-1+\eps}),\\\label{asphimq1p}
\phimq(\xi)&=& O(|\xi|^{\eps})
\eqa
(and $\phimq(\xi)$ is uniformly bounded on the lower half-plane);
\item
if $\al=1$ and $c_+<c_-$, then, for any $\eps>0$, 
\beqa\label{asphipq1m}
\phipq(\xi)&=&O(|\xi|^{\eps}),\\\label{asphimq1m}
\phimq(\xi)&=& O(|\xi|^{-1-\eps})
\eqa
(and $\phipq(\xi)$ is uniformly bounded on the upper half-plane).
\end{enumerate}
\end{enumerate}
\end{lem}

 \begin{cor}\label{cor:asphipq}
 Let $X$ be either a stable L\'evy process of index $\al\neq 1$ or index $\al=1$ and $c_+=c_-$.
 Then, for any $q>0$, there exist $a^\pm_q\ge 0$ and $\de_\pm>0$ such that, for any $\ga\in (-\pi/2,\pi/2)$,
 \bbe\label{eq:asqpmq}
 \phi^\pm_q(\xi)=a^\pm_q+O(|\xi|^{-\de_\pm}), \ (\cC_\ga\ni)\xi\to \infty.
 \ee 
\end{cor}  
\begin{rem}\label{rem:errorFT}{\rm We use \eq{asphim0} and \eq{eq:asqpmq} to regularize the action of the EPV operators $\cE^\pm_q$ as PDO: $\cE^\pm_q u(x)=\cF^{-1}_{\xi\to x}\phi^\pm_q(\xi)\cF_{x\to\xi} u(x)$. Without the regularization, the integrands in the integral representation  that we derive decay too slowly at infinity, and accurate numerical realizations require unnecessarily long grids.  
 }
 \end{rem}
 For $\be >0$, introduce   the convolution operators  $I^\pm_{\be}$ by 
\beqa\label{def:Ip}
I^+_\be u(x)&=&\be\int_{0}^{\infty} e^{-\be y}u(x+y)dy\\\label{def:Im}
I^-_{\be} u(x)&=&\be \int_{-\infty}^{0} e^{\be y}u(x+y)dy.
\eqa
The symbol of $I^\pm_{\be}$ is $(1\mp i\xi/\be)^{-1}$. The following proposition is evident.
\begin{lem}\label{lem_supp} Let $a_1, a_2\in \bC$, $\be>0$ and let there exist $\be_1<\be$ such that $e^{-\be_1|\cdot|}u\in L_1$. Then
\begin{enumerate}[(a)]
\item
if $u(x)=0$ for all $x< h$ (resp., $x\le h$), then $((a_1I+a_2I^-_{\be})u)(x)=0, x<h$ (resp., $x\le h$);
\item
if $u(x)=0$ for all $x> h$ (resp., $x\ge h$), then $((a_1I+a_2I^+_{\be})u)(x)=0, x>h$ (resp., $x\ge h$).

\end{enumerate}
\end{lem}
We calculate the integrals on the RHS' of the formulas \eq{phip1}-\eq{phim1}  for the Wiener-Hopf factors and \eq{amq12}, \eq{apq12}, \eq{amq}, 
\eq{apq} for the constants $a^\pm_q$ representing each integral as a sum of integrals over two rays of the form
$e^{i\ga}\bR_+$. In the integral over a ray of the form $e^{i\ga}\bR_+$, we change the variable $\eta=e^{i\ga+y}$ and apply the simplified trapezoid rule. In the process of deformation, the expressions under the logarithm sign must be well-defined and may not assume values in $(-\infty,0]$.

\subsection{Case of complex $q$}\label{ss:form_WHF_complex}
If the conformal deformation technique is applied to the Bromwich integral, then the Wiener-Hopf factors need to be calculated
for $q$ on the (deformed) contour of integration, and additional conditions  on the parameters of the process (see  Lemma \ref{q_xi_cone}(a)) and the deformed  contour have to be imposed. The restrictions on the parameters of the process and the rotation parameter $\om$ in the formulas for the Wiener-Hopf factors and formulas for $\tV(f;q;x_1,x_2)$ are similar to but different from to the restrictions in Section \ref{ss:1D_alt_Br}. 
The condition in Lemma \ref{q_xi_cone} (a)  must hold for $\psi_{st}$ because \eq{asym_1} must hold for $\psi_{st}$ whereas in Section \ref{ss:1D_alt_Br},
we need \eq{asym_1} to hold for $\psi^0_{st}$ and $\psi_{BM}$. 

\begin{lem}\label{lem:anal_cont_WHF}
Let the condition in Lemma \ref{q_xi_cone} (a)  hold. Then there exist $\sg>0, \ga_0\in (0,\pi/2)$ and $\gam<0<\gap$ such that
\begin{enumerate}[(a)]
\item
$\phipq(\xi)$ admits analytic continuation to $\{(q,\xi)\ |\ q\in \sg+\cC_{\pi/2+\ga_0},
\xi\in i\cC_{\pi/2-\gam}\}$. For any $\omm\in (\gam,0)$ and $\ga^{--}\in (\omm,0)$, the restriction of $\phipq(\xi)$ on $\{(q,\xi)\ |\ q\in \sg+\cC_{\pi/2+\ga_0},
\xi\in i\cC_{\pi/2-\ga^{--}}\}$ is defined by 
 \eq{phip1};  
\item
$\phimq(\xi)$ admits analytic continuation to $\{(q,\xi)\ |\ q\in \sg+\cC_{\pi/2+\ga_0},
\xi\in -i\cC_{\pi/2+\gap}\}$. For any $\omp\in (0,\gap)$ and $\ga^{++}\in (0,\omp)$, the restriction of $\phimq(\xi)$ on $\{(q,\xi)\ |\ q\in \sg+\cC_{\pi/2+\ga_0},
\xi\in -i\cC_{\pi/2+\ga_{++}}\}$ is defined by 
 \eq{phim1};  
\item
if $\al\in (1,2)$ and $c_-=0$, then $a^-_q$ defined by \eq{amq12} is an analytic function on  $\sg+\cC_{\pi/2+\ga_0}$;
\item
if $\al\in (1,2)$ and $c_+=0$, then $a^+_q$ defined by \eq{apq12} is an analytic function on  $\sg+\cC_{\pi/2+\ga_0}$.
\end{enumerate}
\end{lem}
\begin{proof} (a) By Lemma \ref{q_xi_cone}, the condition in Lemma \ref{q_xi_cone} (a)  implies that there exist $\sg, \ga_0, -\pi/2<\gam<0<\gap<\pi/2$ such that \eq{asym_1} holds. Take $\omm\in (\gam,0)$, and, for $q\in \sg+\cC_{\pi/2+\ga_0}$ and $\xi$ above $\cL^-_{\omm}$, define $\phipq(\xi)$ by \eq{phip1}. It follows from \eq{asym_1} that $\phipq(\xi)$ is analytic on $\{q\in \sg+\cC_{\pi/2+\ga_0},
\xi\in i\cC_{\pi/2-\omm}\}$. It remains to note that $\omm\in (\gam,0)$ is arbitrary.

(b) The proof is by symmetry.

(c,d) If $\psi_{st}$ satisfies the condition in Lemma \ref{q_xi_cone} (a), then  $\psi^0_{st}$ satisfies this condition  as well. Hence, there exist $\sg, \ga_0, -\pi/2<\gam<0<\gap<\pi/2$ such that   \eq{asym_1} holds for $\psi_{st}$ and $\psi^0_{st}$.

\end{proof}

\section{Integral representations}\label{int_repr}
In the remainder of the main body of the paper, we assume that either  $\al\neq 1$ or $\al=1$ and $c_+=c_-$, and we use the representation \eq{eq:asqpmq}. We exclude the case of asymmetric stable L\'evy processes of order $\al=1$ because
of the irregular behavior of the Wiener-Hopf factors. 

\subsection{Calculation of the cpdf of $\barX_T$}\label{ss:cpdf_barXT}
Let $a\ge x$, and $T>0$. We calculate 
\bbe\label{cpdfsup0}
V(T,a;x)=\bE[\bfo_{x+\barX_T\le a}]=1-\bE[\bfo_{x+\barX_T> a}]=1-\bE[\bfo_{\barX_T> a-x}].
\ee
Take $\sg>0$ and use the Bromwich integral: 
\beqa\label{VcpdfbarX}
\bE[\bfo_{x+\barX_T> a}]&=&\frac{1}{2\pi i}\int_{\Re q=\sg}q^{-1}(\cEpq \bfo_{(a,+\infty)})(x). 
\eqa
On the strength of Lemma \ref{lem_supp} (a), we can replace $\cEpq$ with $\cEpq-(a^+_qI+(1-a^+_qI^-_1)$. The symbol of the modified operator
$\phi^+_{q,mod}(\xi_1):=\phipq(\xi_1)-(a^+_q+(1-a^+_q)/(1+i\xi_1))$ is $O(|\xi_1|^{\min\{1,\al\}})$ as $\xi_1\to 0$, and $O(|\xi_1|^{-\de_+})$ as $\xi_1\to\infty$, where $\de_+>0$ is from \eq{eq:asqpmq},
and $\widehat{\bfo_{(a,+\infty)}}(\xi_1)=e^{-ia\xi_1}/(i\xi_1)$.
Therefore, for $x\le a$,
\bbe\label{cpdfsup}
(\cEpq \bfo_{(a,+\infty)})(x)=\frac{1}{2\pi}\int_\bR e^{i(x-a)\xi_1}\frac{\phipq(\xi_1)-(a^+_q+(1-a^+_q)(1+i\xi_1)^{-1})}{i\xi_1}d\xi_1,
\ee where for any $\ga\in (0,\pi/2)$, the integrand 
is uniformly bounded on $\{\mathrm{arg}\, \xi_1\in [-\pi, -\pi+\ga]\cup [-\ga,0]\}$ and
 decays as $|\xi_1|^{-1-\de_+}$ as
$\xi_1\to \infty$ in $\{\xi_1\ |\ \mathrm{arg}\, \xi_1\in [-\pi, -\pi+\ga]\cup [-\ga,0]\}$.

\subsubsection{Numerical realization based on GWR algorithm}\label{sss:cpdf_barXT_numer_GWR}
For each $q>0$ used in GWR algorithm, we calculate $(\cEpq \bfo_{(a,+\infty)})(x), x\le a$.  Since $x-a\le 0$, we can take any $\omm\in (-\pi/2,0)$, and deform
the contour of integration into $\cL^-_{\omm}=e^{i\omm}\bR_+\cup e^{i(\pi-\omm)}\bR_+$: 
\bbe\label{cpdfsup2}
(\cEpq \bfo_{(a,+\infty)})(x)=\frac{1}{2\pi}\int_{\cL^-_{\omm}} e^{i(x-a)\xi_1}\frac{\phipq(\xi_1)-(a^+_q+(1-a^+_q)(1+i\xi_1)^{-1})}{i\xi_1}d\xi_1.
\ee
We use  $\overline{\phipq(\xi)}=\phipq(-\bar\xi)$ and $\overline{e^{i(x-a)\xi}}=e^{-i(x-a)\bar\xi}$ to obtain
\bbe\label{cpdfsup3}
(\cEpq \bfo_{(a,+\infty)})(x)=\frac{1}{\pi}\Re\int_{e^{i\omm}\bR_+} e^{i(x-a)\xi}\frac{\phipq(\xi)-(a^+_q+(1-a^+_q)(1+i\xi)^{-1})}{i\xi}d\xi.
\ee
Then we change the variable $\xi=\xi(y)=e^{i\omm+y}$
\bbe\label{cpdfsup4}
(\cEpq \bfo_{(a,+\infty)})(x)=\frac{1}{\pi}\Im\int_{-\infty}^{+\infty} e^{i(x-a)\xi(y)}
(\phipq(\xi(y))-a^+_q-(1-a^+_q)/(1+i\xi(y)))dy,
\ee
and evaluate the integral on the RHS of \eq{cpdfsup4} using the simplified trapezoid rule. We apply the bounds and constructions analogous to the ones 
in \cite{ConfAccelerationStable}, where the pdf and cpdf of stable distributions were efficiently evaluated. In order to find the number of terms in the simplified trapezoid rule sufficient to achieve the target precision, we use the bound for $e^{i(x-a)\xi(y)}\phi^+_{q,mod}(\xi(y))$: 
\beqa\label{eq:asymp0cpdfminf}
e^{i(x-a)\xi(y)}\phi^+_{q,mod}(\xi(y))&=&O(e^{\min\{\al,1\}y}), \ y\to-\infty,\\\label{eq:asymp0cpdfpinf}
e^{i(x-a)\xi(y)}\phi^+_{q,mod}(\xi(y))&=&O(e^{-\min\{\de,1\}y}\cdot e^{\sin(\omm)(a-x)e^y}), \ y\to+\infty,
 \eqa
 where  $\de>0$ is the same as in in Lemma \ref{asympWHF}. In particular,
 if \eq{asphipmqmain} holds, $\de=\alp$.
 
 \subsubsection{Numerical realization based on the sinh-acceleration applied to the Bromwich integral}\label{sss:cpdf_barXT_numer_sinh}
 The sinh-acceleration can be applied to the Bromwich integral only if analytic continuation of $\phipq(\xi)$ w.r.t. $q$ is possible, and if we deform the contour in the formula for $(\cEpq \bfo_{(a,+\infty)})(x), x\le a$, then analytic continuation w.r.t $(q,\xi)$ must be possible.
 Assume that the condition in Lemma \ref{q_xi_cone} (a) holds. Using Lemma \ref{lem:anal_cont_WHF}, we can 
find $\sg, \ga_0, -\pi/2<\gam<0<\gap<\pi/2$ such that \eq{asym_1} holds. Then choose $\sg_\ell>\sg$,
$\om_\ell\in (0,\ga_0)$, next, choose $b_\ell>0$ such that  $\sg_\ell-b_\ell\sin\om_\ell>0$ and define the  map
$\chi_{L; \sg_\ell,b_\ell,\om_\ell}$ by \eq{eq:sinhLapl}. Choose $\omm\in (\gam,0)$. For $q=q(y')=\chi_{L; \sg_\ell,b_\ell,\om_\ell}(y')$
 on the contour $\cL^{(L)}=\chi_{L; \sg_\ell,b_\ell,\om_\ell}(R)$, and $\xi\in \cL^-_{\omm}$, 
we calculate 
 $\phi^\pm_q(\xi)$ and $a^\pm_q$ as in Lemma \ref{lem:anal_cont_WHF}, and then
 \beqa\label{VcpdfbarXsinh}
\bE[\bfo_{x+\barX_T> a}]&=&\frac{b_\ell}{2\pi }\int_{\cL^{(L)}}dy'\,e^{q(y')T}\frac{\cosh(i\om_\ell +y')}{q(y')}(\cE^+_{q(y')} \bfo_{(a,+\infty)})(x),
\eqa
where $(\cE^+_{q(y')} \bfo_{(a,+\infty)})(x)$ is given by \eq{cpdfsup2}. We apply the simplified trapezoid rule to the integral on the RHS of
\eq{VcpdfbarXsinh}, and, for each $y'$ on the grid, evaluate 
\[
(\cE^+_{q(y')} \bfo_{(a,+\infty)})(x)=\sum_{\ga=\omm, \pi-\omm}I_\ga(y',x),
\]
where 
\bbe\label{cpdfsup5}
I_\ga(y',x)=\frac{1}{2\pi}\int_{e^{i\ga}\bR_+} e^{i(x-a)\xi_1}
\frac{\phi^+_{q(y')}(\xi_1)-(a^+_{q(y')}+(1-a^+_{q(y')})(1+i\xi_1)^{-1})}{i\xi_1}d\xi_1
\ee
(for each $\ga$, the direction is from the left to the right), making the change of variables $\xi=e^{i\ga +y}$ on the RHS of \eq{cpdfsup5}, and then applying the simplified trapezoid rule.
 
 \subsection{Calculation of the joint cpdf of $X_T$ and $\barX_T$}\label{ss:joint_stable}
Let $a_2>0, a_1\le a_2$, $x_1\le x_2\le a_2$,  and $T>0$. Consider 
\[V(T,a_1,a_2;x_1,x_2)=
\bE[\bfo_{(-\infty,a_1]}(x_1+X_T)\bfo_{(-\infty,a_2]}(\max\{x_2, x_1+\barX_T\})].
\] If $a_1=a_2$, we have the cpdf of $\barX_T$,
which we considered in Section \ref{ss:cpdf_barXT}. Hence, we assume that $a_1<a_2$. We have
$
V(T,a_1,a_2;x_1,x_2)=\bE[\bfo_{(-\infty,a_1]}(x_1+X_T)]-V_1(T,a_1,a_2;x_1,x_2)$,
where \[V_1(T,a_1,a_2;x_1,x_2)=\bE[\bfo_{(-\infty,a_1]}(x_1+X_T)\bfo_{(a_2,+\infty)}(\max\{x_2, x_1+\barX_T\})].\]
Function $a_1\mapsto \bE[\bfo_{(-\infty,a_1]}(x_1+X_T)]$ is the cpdf of the stable L\'evy process starting at $x_1$, for which we developed explicit efficient algorithms in \cite{ConfAccelerationStable} and Section \ref{s:1D}. 
Hence, it suffices to consider $V_1(T,a_1,a_2;x_1,x_2)$. Using \eq{perp_opt_X_barr2}, we obtain
\[
\tV_1(q,a_1,a_2;x_1,x_2)=q^{-1}(\cEpq\bfo_{(a_2,+\infty)}\cEmq  \bfo_{(-\infty,a_1]})(x_1).
\]
Since $a_2>a_1$ and $x_1\le a_2$, Lemma \ref{lem_supp} allows us to replace 
$\bfo_{(a_2,+\infty)}\cEmq  \bfo_{(-\infty,a_1]}$ and $\cEpq\bfo_{(a_2,+\infty)}$
with $\bfo_{(a_2,+\infty)}\cE^{-}_{q,mod}\bfo_{(-\infty,a_1]}$ and   
$\cE^{+}_{q,mod}\bfo_{(a_2,+\infty)}$, respectively, where 
\[
\cE^{-}_{q,mod}=\cEmq-(a^-_qI+(1-a^-_q)I^+_1), \
\cE^{+}_{q,mod}=\cEpq-(a^+_qI+(1-a^+_q)I^-_1).
\]
We have $\widehat{\bfo_{(-\infty,a_1]})}(\xi_1)=e^{-ia_1\xi_1}/(-i\xi_1)$, and the symbols of the modified EPV operators
$\phi^+_{q,mod}(\xi_1):=\phipq(\xi_1)-(a^+_q+(1-a^+_q)/(1+i\xi_1))$ and 
$\phi^-_{q,mod}(\xi_1):=\phimq(\xi_1)-(a^-_q+(1-a^-_q)/(1-i\xi_1))$
are $O(|\xi_1|^{\min\{1,\al\}})$ as $\xi_1\to 0$, and $O(|\xi_1|^{-\de_\pm})$ as $\xi_1\to\infty$, where $\de_\pm>0$ are from \eq{eq:asqpmq}.
Therefore, we can use the pdo representation of the modified EPV operators
to calculate, for $y\ge a_2$,
\[
(\cE^{-}_{q,mod}\bfo_{(-\infty,a_1]})(y)=\frac{1}{2\pi}\int_{\bR}e^{i(y-a_1)\xi_1}\frac{\phi^-_{q,mod}(\xi_1)}{-i\xi_1}d\xi_1.
\]
Let $q>0$. Since $y-a_1>0$, we can deform the contour of integration: for any $\omp\in (0,\pi/2)$,
\[
(\cE^{-}_{q,mod}\bfo_{(-\infty,a_1]})(y)=\frac{1}{2\pi}\int_{\cL^+}e^{i(y-a_1)\xi_1}\frac{\phi^-_{q,mod}(\xi_1)}{-i\xi_1}d\xi_1.
\]
where $\cL^+=\cL^+_{\omp}=e^{i\omp}\bR_+\cup e^{i(\pi-\omp)}\bR_+$. (In the case of $q\not\in (0,+\infty)$, we need to impose the condition in Lemma \ref{q_xi_cone} (a), and choose
$\omp$ so that $q+\psi(\xi_1)\neq 0$ for all $q$ and $\xi_1$ arising in the process of deformation).
Then we calculate the Fourier transform $\hat w(\eta)$ of $w:=\bfo_{(a_2,+\infty)}\cE^{-}_{q,mod}\bfo_{(-\infty,a_1]}$ for 
$\{\eta\ |\ \Im\eta\le 0, \eta\neq 0\}$;  since the integrand below decays exponentially as a function of $y$ as $y\to+\infty$, and as $|\xi|^{-1-\de_-}$ as $\xi_1\to\infty$, where $\de_->0$, Fubini's theorem is applicable, and we obtain
\beqast
\hat w(\eta)&=&\int_{a_2}^{+\infty} dy e^{-iy\eta}\frac{1}{2\pi}\int_{\cL^+}e^{i(y-a_1)\xi_1}\frac{\phi^-_{q,mod}(\xi_1)}{-i\xi_1}d\xi_1\\
&=&e^{-ia_2\eta}\frac{1}{2\pi}\int_{\cL^+}e^{i(a_2-a_1)\xi_1}\frac{\phi^-_{q,mod}(\xi_1)}{(-i\xi_1)i(\eta-\xi_1)}d\xi_1.
\eqast
Let $q>0$. Fix $\omm\in (-\pi/2,0)$, and introduce 
$\Om=\{(\xi_1,\eta)\ |\ \xi_1\in (\cL^+\setminus\{0\}), \mathrm{arg}\, \eta\in [-\pi,-\pi-\omm]\cup [\omm,0]\}$
and $\cL^-=\cL^-_{\omm}=e^{i\omm}\bR_+\cup e^{i(\pi-\omm)}\bR_+$.
Since $x_1\le a_2$, the absolute value of the integrand of the double integral 
\[\tV_1(q,a_1,a_2;x_1,x_2)
=\frac{1}{(2\pi)^2 q}\int_\bR d\eta\, e^{i(x_1-a_2)\eta}\phi^+_{q,mod}(\eta)\int_{\cL^+}e^{i(a_2-a_1)\xi_1}\frac{\phi^-_{q,mod}(\xi_1)}{\xi_1(\eta-\xi_1)}d\xi_1
\]
 is bounded by 
$C(q,\omm,\omp)g(|\xi_1|,|\eta|)$, where $C(q,\omm,\omp)$ is independent of  $(\xi_1,\eta)\in \Om$ and
\[
g(|\xi_1|,|\eta|)=(|\eta|^{-\de_+}\bfo_{|\eta|\ge 1}+|\eta|^{\min\{1,\al\}}\bfo_{|\eta|\le 1})(|\eta|+|\xi|)^{-1}
(|\xi_1|^{-1-\de_-}\bfo_{|\xi_1|\ge 1}+|\xi_1|^{\min\{0,\al-1\}}\bfo_{|\xi_1|\le 1}).
\]
Considering separately the subregions defined by the pairs of inequalities (1) $|\xi_1|\le1, |\eta|\le 1$; (2)
$|\xi_1|\ge1, |\eta|\le 1$; (3) $|\xi_1|\le1, |\eta|\ge 1$; (4) $|\xi_1|\ge1, |\eta|\ge 1$; and taking into account that 
$\de_\pm>0$ and $\al>0$, we conclude that
$g\in L_1(\bR^2)$. Hence, we can change the order of integration and obtain 
\bbe\label{tV1q} \tV_1(q,a_1,a_2;x_1,x_2)=\frac{1}{(2\pi)^2 q}\int_{\cL^-} d\eta\, e^{i(x_1-a_2)\eta}\phi^+_{q,mod}(\eta)
\int_{\cL^+}\frac{e^{i(a_2-a_1)\xi_1}\phi^-_{q,mod}(\xi_1)}{\xi(\eta-\xi_1)}d\xi_1.
\ee

\subsection{General case and example}\label{int_repr_gen} We evaluate $V(f;T;x_1,x_2)=\bE[f_+(x_1+X_T, \max\{x_2, x_1+\barX_T\})]$.
The first term on the RHS of \eq{tVq0} is calculated at the end of Section \ref{ss:WHF_gen}.  Below, we consider in detail the evaluation of the second term on the RHS of \eq{tVq0}.
The proof of the following theorem is a straightforward modification of the proof for the joint cpdf. We assume that either $q>0$ or the condition in Lemma \ref{q_xi_cone} (a) holds,
 and the parameters of the deformations are chosen so that $q+\psi_{st}(\xi_1), q+\psi_{st}(\eta)\not\in (-\infty,0]$ for all $(q,\xi_1,\eta)$ arising in the process of deformation of all contours of integration.
\begin{thm}\label{thm:gen_stable}
Let $x_1\le x_2$, and let the following conditions hold
\begin{enumerate}[(i)]
\item
$\exists$\ $\al_0>-\al$, $ \de<\de_-$, $\gap\in (0,\pi/2)$ such that $\forall$\ 
$\xi_1\in\{\xi_1\ |\  \mathrm{arg}\,\xi_1\in [0,\gap]\cup[\pi,\pi-\gap]\}$,
\bbe\label{boundf1}
|e^{ix_2\xi_1}\widehat{(f_+)}(\xi_1,x_2)|\le C_1(x_2)(|\xi_1|^{\al_0-1}\bfo_{|\xi_1|\le 1}+|\xi_1|^{-1+\de}\bfo_{|\xi_1|\ge 1}),
\ee 
where $C_1(x_2)$ depends on $x_2$ but not on $\xi_1$;
\item
$\exists$\ $\al_0>-\al$, $ \de<\de_+$, $\gam\in (-\pi/2,0)$ such that $\forall$\  
$\eta\in\{\eta\ |\ \mathrm{arg}\,\eta\in [\gam,0]\cup[-\pi,-\pi-\gam]\}$,
\bbe\label{boundw1}
\left|e^{ix_2\eta}\int_{x_2}^{+\infty} dy\, e^{iy\eta}((\cE^-_{q,mod}\otimes I)f_+)(y,y)\right|\le C_2(x_2)(|\eta|^{\al_0-1}\bfo_{|\eta|\le 1}+|\eta|^{-1+\de}\bfo_{|\eta|\ge 1}),
\ee 
where $C_2(x_2)$ depends on $x_2$ but not on $\eta$.
\end{enumerate}
Then, for any $x_1\le x_2$, $\omm\in (\gam,0)$ and $\omp\in (0,\gap)$, the second term on the RHS of \eq{tVq0}
admits the representation
\beqa\label{eq:gen_stable}
&&(\cE^+_{q,mod}\bfo_{(x_2,+\infty)}(y)(((\cE^-_{q,mod}\otimes I)f_+)(y,y)-
((\cE^-_{q,mod}\otimes I)f_+)(y,x_2)))(x_1)\\\nonumber
&=&\frac{1}{2\pi}\int_{\cL^-_{\omm}} d\eta\, e^{i(x_1-x_2)\eta}\phi^+_{q,mod}(\eta)
\int_{x_2}^{+\infty} dy\, e^{i(x_2-y)\eta}\\\nonumber
&&\cdot\frac{1}{2\pi}\int_{\cL^+_{\omp}} e^{iy\xi_1}\phi^-_{q,mod}(\xi_1)
(\widehat{(f_+)}_1(\xi_1,y)-\widehat{(f_+)}_1(\xi_1,x_2)) d\xi_1.
\eqa
\end{thm}

\begin{example}\label{ex:exchange_stable}{\rm For $\be>1$, $\la>0$, set 
$f(x_1,x_2):=f(\la,\be;x_1,x_2)=(\be x_1-x_2)_+ \bfo_{x_1\le x_2}e^{-\la x_2}$. We introduce the dampening factor $e^{-\la x_2}$ so that
$V(f(\la,\be;\cdot,\cdot);T;x_1,x_2)$ be finite for any $\al\in (0,2)$. By inspection of the integral representation for $\tV(f,q,x_1,x_2)$ 
that we derive, it is seen that if $\al>1$, the integrands are uniformly bounded by an absolutely integrable
function independent of $\la$, and converge point-wise as $\la\to0$. Hence,
the limits $\tV(f(0+,\be;\cdot,\cdot); q,x_1,x_2)$ and $V(f(0+,\be;\cdot,\cdot); T,x_1,x_2)$  exist and and are given by
the same integrals with $\la=0$. 

The function
\beqast
\widehat{(f_+)}_1(\xi_1,x_2)&=&e^{-\la x_2}\int_{x_2/\be}^{x_2} e^{-ix_1\xi_1}(\be x_1-x_2)dx_1\\
&=&\frac{e^{-\la x_2}}{-i\xi_1}\left[(\be-1)x_2e^{-ix_2\xi_1}-\be\int_{x_2/\be}^{x_2} e^{-ix_1\xi_1}d\xi_1\right]\\
&=&e^{-\la x_2}\left([(-i\xi_1)^{-1}(\be-1)x_2+\xi_1^{-2}\be] e^{-ix_2\xi_1}-\be\xi_1^{-2}e^{-ix_2\xi_1/\be}\right)
\eqast
satisfies \eq{boundf1} with $\al_0=1, \de=0$.  
For $y>x_2$, $e^{i(y-x_2)\xi_1}$ is uniformly bounded in
the upper half-plane, and exponentially decays as $\xi_1\to \infty$ along any ray $e^{i\omp}\bR_+$, $\omp\in (0,\pi)$. Since 
$1/\be<1$, $e^{i(y-x_2/\be)\xi_1}$ enjoys the same property since $x_2\ge 0$.
Let $q>0$. We take $\omp\in (0,\pi/2)$, and calculate 
\beqast 
&&\cF_{y\to\eta}\bfo_{(x_2,+\infty)}(y)((\cE^-_{q,mod}\otimes I)f_+)(y,x_2)\\
&=&\int_{x_2}^{\infty}dy\, e^{-(\la+i\eta)y}\frac{1}{2\pi}\int_\bR \phi^-_{q,mod}(\xi_1)\left[e^{i(y-x_2)\xi_1}
\left(\frac{(\be-1)x_2}{-i\xi_1}+\frac{\be}{\xi_1^2}\right)-e^{i(y-x_2/\be)\xi_1}\frac{\be}{\xi_1^2}\right] d\xi_1\\
&=&\int_{x_2}^{\infty}dy\, e^{-(\la+i\eta)y}\frac{1}{2\pi}\int_{\cL^+_{\omp}} \phi^-_{q,mod}(\xi_1)\left[e^{i(y-x_2)\xi_1}
\left(\frac{(\be-1)x_2}{-i\xi_1}+\frac{\be}{\xi_1^2}\right)-e^{i(y-x_2/\be)\xi_1}\frac{\be}{\xi_1^2}\right] d\xi_1
\eqast
For $\eta\in \{\Im\eta\le 0, \eta\neq 0\}$, the integrand decays exponentially w.r.t. $\xi_1$ and $y$, hence, the double integral is absolutely convergent. Applying Fubini's theorem, and calculating the integral w.r.t. $y$, we continue
\beqast
&=&\frac{e^{-(\la+i\eta)x_2}}{2\pi}\int_{\cL^+_{\omp}} \frac{\phi^-_{q,mod}(\xi_1)}{\la+i(\eta-\xi_1)}\left[
\left(\frac{(\be-1)x_2}{-i\xi_1}+\frac{\be}{\xi_1^2}\right)-e^{ix_2(1-1/\be)\xi_1}\frac{\be}{\xi_1^2}\right] d\xi_1.
\eqast
Similarly,
\beqast
&&\cF_{y\to\eta}\bfo_{(x_2,+\infty)}(y)((\cE^-_{q,mod}\otimes I)f_+)(y,y) \\
&=&\int_{x_2}^{\infty}dy\, e^{-(\la+i\eta)y}\frac{1}{2\pi}\int_{\cL^+_{\omp}} \phi^-_{q,mod}(\xi_1)\left[
\left(\frac{(\be-1)y}{-i\xi_1}+\frac{\be}{\xi_1^2}\right)-e^{iy(1-1/\be)\xi_1}\frac{\be}{\xi_1^2}\right] d\xi_1\\
&=&\frac{e^{-(\la+i\eta) x_2}}{2\pi}\int_{\cL^+_{\omp}} \phi^-_{q,mod}(\xi_1)\left[
\left(\left(\frac{1}{\la+i\eta}-\frac{1}{(\la+i\eta)^2}\right)\frac{(\be-1)}{-i\xi_1}+\frac{1}{\la+i\eta}\frac{\be}{\xi_1^2}\right) 
\right.\\&& \hskip2cm\left. 
-\frac{\be e^{ix_2(1-1/\be)\xi_1}}{\xi_1^2(\la+i(\eta-(1-1/\be)\xi_1))}\right] d\xi_1.
\eqast
Substituting into \eq{eq:gen_stable}, we obtain
\beqa\label{eq:gen_stable_example}
&&(\cE^+_{q,mod}\bfo_{(x_2,+\infty)}(y)(((\cE^-_{q,mod}\otimes I)f_+)(y,y)-
((\cE^-_{q,mod}\otimes I)f_+)(y,x_2)))(x_1)\\\nonumber
&=&\frac{e^{-\la x_2}}{(2\pi)^2}\int_{\cL^-_{\omm}} d\eta\, e^{i(x_1-x_2)\eta}\phi^+_{q,mod}(\eta)
\int_{\cL^+_{\omp}}d\xi_1\,\phi^-_{q,mod}(\xi_1)\\\nonumber
&&\cdot\left[\left(\frac{1}{\la+i\eta}-\frac{1}{(\la+i\eta)^2}\right)\frac{(\be-1)}{-i\xi_1}-\frac{(\be-1)x_2}{-i\xi_1(\la+i(\eta-\xi_1)}
+\frac{\be}{i\xi_1(\la+i\eta)(\la+i(\eta-\xi_1))}\right.\\\nonumber
&&\left. +\frac{\be e^{ix_2(1-1/\be)\xi_1}(1-1/\be)}{-i\xi_1(\la+i(\eta-(1-1/\be)\xi_1)(\la+i\eta)}\right]\\\nonumber
&=&\frac{e^{-\la x_2}}{(2\pi)^2}\int_{\cL^-_{\omm}} d\eta\, e^{i(x_1-x_2)\eta}\phi^+_{q,mod}(\eta)
\int_{\cL^+_{\omp}}d\xi_1\,\frac{\phi^-_{q,mod}(\xi_1)}{-i\xi_1}\\\nonumber
&&\cdot\left[\left(\frac{1}{\la+i\eta}-\frac{1}{(\la+i\eta)^2}\right)(\be-1)-\frac{(\be-1)x_2}{\la+i(\eta-\xi_1)}
-\frac{\be}{(\la+i\eta)(\la+i(\eta-\xi_1))}\right.\\\nonumber
&&\left. +\frac{\be e^{ix_2(1-1/\be)\xi_1}(1-1/\be)}{(\la+i(\eta-(1-1/\be)\xi_1)(\la+i\eta)}\right]\
\eqa

}
\end{example}

\section{Numerical evaluation of the cpdf of the supremum process and joint cpdf}\label{numer_real}

\subsection{Cpdf of the supremum process}\label{ss:supremum_numer:main_blocks}
Evaluation of  $\bE[\bfo_{x+\barX_T\le a}]$ is reducible to evaluation of \eq{cpdfsup4} (see \eq{cpdfsup0}). The integral on
the RHS of  \eq{cpdfsup4} is evaluated  using the simplified trapezoid rule exactly as the outer integral in the case of the joint cpdf is evaluated in Sect. \ref{ss:joint_stable_numer}.  The Laplace inversion is performed exactly as in the case of the joint cpdf.

 \subsection{Joint cpdf of $X_T$ and $\barX_T$: numerical realization based on GWR algorithm}\label{ss:joint_stable_numer}
 We use $V(T,a_1,a_2;x_1,x_2)=\bE[\bfo_{(-\infty,a_1]}(x_1+X_T)]-V_1(T,a_1,a_2;x_1,x_2)$. Detailed algorithms for and numerous examples of numerical evaluation of $\bE[\bfo_{(-\infty,a_1]}(x_1+X_T)]$ can be found in
 \cite{ConfAccelerationStable}, therefore, it suffices to provide an algorithm for numerical evaluation of $V_1$; the Laplace transform
 $ \tV_1(q,a_1,a_2;x_1,x_2)$ is given by \eq{tV1q}. We proved \eq{tV1q} assuming that  
  either $\al\neq 1$ or $\al=1$ and $c_-=c_+$. If $\al=1$ and $c_+\neq c_-$, then one can use \eq{tV1q} with $\phi^\pm_q$ instead
  of the modified $\phi^\pm_{q,mod}$ but then the convergence of the integrals and infinite sums in the simplified trapezoid rule 
  are worse, longer grids are needed, and errors are more difficult to control.

 Set $\omp=\pi/8, \omm=-\pi/8$, construct grids
  \beqa\label{plus:grids}
 \vec{y^{+}}&=&\ze^+*(-N^+_-, N^+_+),\ \vec{\xi^{+-}}=\exp[i*(\pi-\omp)+ \vec{y^{+}}],\ \vec{\xi^{++}}=\exp[i*\omp+ \vec{y^{+}}],
 \\\label{minus:grids}
 \vec{y^{-}}&=&\ze^-*(-N^-_-, N^-_+),\ \vec{\xi^{--}}=\exp[i*(-\pi-\omm)+ \vec{y^{-}}],\ \vec{\xi^{-+}}=\exp[i*\omm+ \vec{y^{-}}];
 \eqa
and
 calculate the arrays $\psi_{st}(\vec{\xi^{+-}}), \psi_{st}(\vec{\xi^{++}}), \psi_{st}(\vec{\xi^{--}}), \psi_{st}(\vec{\xi^{-+}})$. 
 
 Then,
 for each $q>0$ used in GWR algorithm, calculate
 \begin{enumerate}[(I)]
  \item
 the Wiener-Hopf factors $\phimq(\vec{\xi^{--}})$ and $\phimq(\vec{\xi^{-+}})$ and $a^-_q\ge 0$ using the formulas in Lemma \ref{asympWHF}; the integrals over the rays $e^{i\omp}\bR_+$ and $e^{i(\pi-\omp)}\bR_+$ are calculated using the exponential changes
 of variables and simplified trapezoid rule, the grids \eq{plus:grids} being used;
 \item 
 the Wiener-Hopf factors $\phipq(\vec{\xi^{+-}})$ and $\phipq(\vec{\xi^{++}})$ and $a^+_q\ge 0$ using the formulas in Lemma \ref{asympWHF}; the integrals over the rays $e^{i\omm}\bR_+$ and $e^{i(-\pi-\omm)}\bR_+$ are calculated using the exponential changes
 of variables and simplified trapezoid rule, the grids \eq{minus:grids} being used;
 \item
 arrays $\phipq(\vec{\xi^{--}})$, $\phipq(\vec{\xi^{-+}})$ and $\phimq(\vec{\xi^{+-}})$, $\phimq(\vec{\xi^{++}})$ using the arrays
 in (I) and (II) and the Wiener-Hopf identity;
 \item
 arrays  $\phi^+_{q,mod}(\vec{\xi^{--}})$, $\phi^+_{q, mod}(\vec{\xi^{-+}})$, 
 $\phi^-_{q,mod}(\vec{\xi^{+-}})$ and $\phi^-_{q, mod}(\vec{\xi^{++}})$;
 \item
 making the exponential changes of variables on each of four arrays comprising the contours in the double integral on the right-most side  of
 \eq{tV1q} and applying the simplified trapezoid rule to each of 4 integrals, calculate $\tV_1(q,a_1,a_2;x_1,x_2)$.
  \end{enumerate}
  Finally, apply GWR algorithm to calculate $V_1(T,a_1,a_2;x_1,x_2)$. 
  
 
 \subsection{Sinh-acceleration in the Bromwich integral}\label{ss:SINH-Bromwich}
Let the condition in Lemma \ref{q_xi_cone} (a) hold. Then we find $\sg>0, \ga_0\in (0,\pi/2)$ and $\gam<0<\gap$ as indicated in Lemma \ref{lem:anal_cont_WHF}.
Next, we choose  $\sg_\ell>\sg$,
$\om_\ell\in (0,\ga_0)$, and $b_\ell>0$ such that  $\sg_\ell-b_\ell\sin\om_\ell>0$ and define the  map
$\chi_{L; \sg_\ell,b_\ell,\om_\ell}$ by \eq{eq:sinhLapl}. In the Bromwich integral for $V_1(T, a_1,a_2;x_1,x_2)$, we make the change
of variables $q=q(y')=\chi_{L; \sg_\ell,b_\ell,\om_\ell}(y')$, and choose the grid $\vec{y'}=(j\ze)_{j=-N}^N$ for the simplified trapezoid rule.
We calculate the factors $\tV_1(q(j\ze),a_1,a_2;x_1,x_2)$ in the approximation
\bbe\label{sinhBrCPDF}
V_1(T,a_1,a_2;x_1,x_2)\approx \frac{b_\ell\ze}{\pi}\Re \sum_{j=0}^N e^{q(j\ze)T}\cosh(i\om_\ell+j\ze))
\tV_1(q(j\ze),a_1,a_2;x_1,x_2)(1-\de_{j0}/2)
\ee
as for $q>0$ above using $\omp=\gap/2$ and $\omm=\gam/2$ and Lemma \ref{lem:anal_cont_WHF}
instead of Lemma \ref{asympWHF}.

 \subsection{Numerical examples}\label{ss:numer_ex}
 The calculations in the paper
were performed in MATLAB 2017b-academic use, on a MacPro Chip Apple M1 Max Pro chip
with 10-core CPU, 24-core GPU, 16-core Neural Engine 32GB unified memory,
1TB SSD storage.

The CPU times shown in the tables in Sect. \ref{s:tables} can be significantly improved 
using 
parallelization, especially in the blocks where calculations are performed for each $q$ in the Laplace inversion procedure, and
 asymptotic expansions and summation by parts (similarly to the ones used to evaluate cpdf of $X_T$ in Sect. \ref{s:1D} and \cite{ConfAccelerationStable}) in the blocks where the Wiener-Hopf factors. Note that when the cpdf of $\barX_T$ is evaluated, the CPU time for the evaluation of the Wiener-Hopf factors is 80\% of the total CPU time, approximately. 
As one would expect, the length of arrays and CPU time depends on parameters of the process, $T$, the barrier $a_2$ and distance to barrier $a_2-a_1$. Hence, the comprehensive list of examples would be extremely long; we choose examples to show that if
the parameters of the process, $T$ and $a_2$, $a_2-a_1$ are in a not very favorable region of the parameter space, the method
of the paper can achieve good precision fairly fast. If the sinh-acceleration is applicable to the Bromwich integral,  we achieve the precision 
better than E-15 in seconds in the case of the joint distribution, and in a fraction of a second for cpdf of the supremum process.
The precision of the order E-10 is achievable in hundreds and dozens of milliseconds, respectively. When GWR acceleration is applied, we fix the order $2M=16$ and choose other parameters of the numerical scheme so that the order of the errors does not improve if we choose longer and finer grids. Thus, the errors shown are, essentially, the errors of GWR acceleration itself. If $\al<1$ and $\mu\neq 0$, then we can apply only GWR acceleration, and we show the differences between the results of two schemes with the integration over different curves. 

As in the case of the evaluation of cpdf of $X_T$, the length of grids and CPU time needed to satisfy a chosen error tolerance greatly 
increases if $\be\neq 0$ and $\al$  is close to 0 or 1; the effect depends on the sign of $\be$. In the examples below, $\al=0.2$ and
$\al=1.2$ are
close but not very close to 0 and 1, and $\be=-0.2$ and $\mu=-0.02$, $\mu=0$ and $\mu=0.02$ are not large in absolute value.
In the majority of tables, $T=0.25$ is moderate, and the upper boundary $a$ for the supremum process is small, in the range $[0.0125, 0.075]$.
In Tables 4-6, $T=0.004, 1, 10$; in Tables 5-6, $a\in [100,600]$ and $a\in [0.5,3]$, respectively. The results for the joint cpdf (Tables 6 and 7) show that
$V_1$ can be calculated more accurately than cpdf of the supremum process using the same grids although the CPU time in the former case (triple integral) is much larger than in the latter case (double integral). The reason is that $V_1$ is smoother than cpdf of the supremum process.

 \section{Conclusion}\label{concl}
 In Sect. \ref{s:1D}, we generalized the  efficient numerical method
 for evaluation of pdfs and cpdfs of stable distributions developed in \cite{ConfAccelerationStable}
 to the case of expectations $\bE[f(x+X_T)$. We also suggested 
 a new method, which can be used for if the parameters of the distribution are in the regions where the method of  \cite{ConfAccelerationStable} requires extremely  grids. 
In the following sections, we derived integral representations for expectations of functions $f(X_T, \barX_T)$ of a stable L\'evy process and the corresponding supremum process and design efficient numerical procedures for the evaluation of the resulting double and triple integrals. 
 The main block is the integral representation for the Laplace transform $\tV(q)$ of the expectation w.r.t. $T$;
 the form of the representation depends on the properties of $f$.
 The integral representations are  analogs of the representations derived in \cite{EfficientLevyExtremum}
 for L\'evy processes with exponentially decaying tails of the density of jumps.
As in  \cite{EfficientLevyExtremum}, the resulting formulas are in terms
 of the Wiener-Hopf factors $\phi^\pm_q(\xi)$. Since the characteristic exponent $\psi$ of a stable L\'evy process is non-smooth at 0,
 the justification of the integral representations in the present paper are more involved, and a novel regularization of the Wiener-Hopf factors 
 in the resulting formulas is needed. The regularization relies on a sufficiently regular behavior of $\phi^\pm_q(\xi)$ as $\xi\to\infty$.
 In the case of asymmetric processes of index 1, we were unable to derive the asymptotics of the Wiener-Hopf factors, hence,
 we have neither a rigorous proof of the integral representations in terms of $\phi^\pm_q(\xi)$ nor  representations 
 in terms of the modified Wiener-Hopf factors. The general theory of pseudo-differential operators (see, e.g., \cite{eskin,NG-MBS})
 can be used to justify the integral representations in the sense of generalized functions, and study the regularity of the integral representations. The numerical schemes that we develop can be used in this case as well but the convergence of the integrals and infinite sums can be poor without the regularization.
 
 Numerical realizations of the formulas for the Wiener-Hopf factors and integral representations are based on appropriate rotations of
 rays of integration, equivalently, on changes of variables of the form $\xi=e^{i\om+y}$; in \cite{EfficientLevyExtremum}, conformal deformations of the lines of integration and the sinh-changes of variables of the form $x\i=i\om_1+b\sinh(i\om+y)$ were used.
In repeated integrals, the changes of variables must be in a certain agreement. 
 As in  \cite{EfficientLevyExtremum}, after the changes of the variables are made, the simplified trapezoid rule is applied.
In  \cite{EfficientLevyExtremum} and the present paper, the complexity of the scheme of the evaluation of the Wiener-Hopf factors is of the order of $O((\ln(1/\eps))^2)$, where $\eps$ is the error tolerance, and  the complexity of the scheme of the evaluation
 of $\tV(q)$ is of the order $O((\ln(1/\eps))^{2m})$, where $m$ is the number of repeated integrations. 
 
 As in \cite{EfficientLevyExtremum}, we consider two types of the Laplace inversion: the Gaver-Wynn Rho acceleration and 
 the sinh-acceleration applied to the Bromwich integral. The latter can be applied in the index $\al>1$ or $\al=1$ in the symmetric case or $\al<1$ and zero drift. Using the sinh-acceleration, it is possible to achieve the precision of the order E-15 in a fraction of a second if
 $V$ is represented as a double integral, and in several seconds if $V$ is represented as a triple integral; the precision of the order of $E-10$ can be achieved in dozens of milliseconds and a fraction of a second, respectively. The scheme based on the Gaver-Wynn Rho acceleration is not so accurate but simpler to implement because for $q>0$, the deformations of the contours of deformations
 enjoy better properties, and the parameters of the deformations are much easier to choose. The reason is the following useful property
 established in \cite{EfficientAmenable} for a wide class of L\'evy processes called Stieltjes-L\'evy processes (SL processes). Namely,
 if $X$ is a SL process, then $q+\psi(\xi)\neq 0$ for all $\xi\not\in i\bR$. It is proved in  \cite{EfficientAmenable} that essentially all popular classes of L\'evy processes, stable ones including, are SL processes.

 The numerical examples in the paper are provided for the cpdf of the supremum process and joint cpdf of the stable L\'evy process and supremum process. As it is well-known (see, e.g., \cite{ConfAccelerationStable} for an extensive list of examples), for certain regions in the parameter space (relatively small ones), the evaluation even of cpdf of the stable L\'evy process is very difficult and requires additional tricks. Similar difficulties are expected in the cases considered in the paper. However, as numerical examples in the paper demonstrate,
 if the set of parameters is not too close to the bad region in the parameter space, the method of the paper works well.

\appendix 

\section{Technicalities}\label{s:techn}
 \subsection{Proof of Lemma \ref{asympWHF}}\label{proof:lem_asympWHF}
   We prove the statements for $\phipq(\xi)$; the proofs for $\phimq(\xi)$ are by symmetry.
 (1) Fix a sufficiently small $\eps>0$ and separate the contour of integration in \eq{phip1} into two parts: 
 $\cL^-_{\om,\eps}=\{\eta\in \cL^-_{\om}\ |\ |\eta|\le \eps\}$ and $\cL^-_\om\setminus \cL^-_{\om,\eps}$. 
 Evidently, the integral over the latter contour is $O(|\xi|)$ as $\xi\to 0$, and the integrand over the former admits the upper bound via
 $Cg(\xi,\eta)$, where $g(\xi,\eta)=|\xi||\eta|^{\al-1}/(|\xi|+|\eta|)$.  We have
 \[
 \int_0^\eps \frac{|\xi||\eta|^{\al-1}}{|\xi|+|\eta|}d|\eta|=|\xi|^{\al}\int_0^{\eps/|\xi|}\frac{|\eta|^{\al-1}}{|\eta|+1}d|\eta|.
 \]
 The last integral is uniformly bounded as $|\xi|\to 0$ if $\al\in (0,1)$; if $\al\in (1,2)$, the integral has the asymptotics
 $(\al-1)^{-1}(\eps/|\xi|)^{\al-1}$, which finishes the proof of \eq{asphim0}.
 
 (2) Define $\phi^{+,1}_q(\xi)=(1-i\xi(q/|\Cp|)^{-1/\al})^{\alp}\phipq(\xi)$, $\phi^{-,1}_q(\xi)=(1+i\xi(q/|\Cp|)^{-1/\al})^{\alm}\phimq(\xi)$,
 \[
 \Phi(q,\eta)=\ln(1+\psi_{st}(\eta)/q)-\alp\ln(1-i\eta/q)-\alm\ln(1+i\eta/q).
 \]
 Using the residue theorem, it is straightforward to prove that
 \beqa\label{phip11}
\phi^{+,1}_q(\xi)&=&\exp\left[\frac{1}{2\pi i}\int_{\cL^-_\om}\frac{\xi \Phi(q,\eta)}{\eta(\xi-\eta)}d\eta\right]\\
\label{phim11}
\phi^{-,1}_q(\xi)&=&\exp\left[-\frac{1}{2\pi i}\int_{\cL^+_\om}\frac{\xi \Phi(q,\eta)}{\eta(\xi-\eta)}d\eta\right].
\eqa
It follows from \eq{aspsimain} that, as $\eta\to\infty$ along $\cL^-_\om$ from 0 to the right, 
\[
\Phi(q,\eta)=i(\varphi_0+\al(-\om)-\alp(-\pi/2-\om)-\alm(\pi/2-\om))+O(|\eta|^{-\de}),
\]
where $\de>0$. Since $\alm+\alp=\al$ and $\varphi_0=(\pi/2)(\alm-\alp)$, we have $\Phi(q,\eta)=O(|\eta|^{-\de})$.
As $\eta\to 0$ along $\cL^-_\om$, $\Phi(q,\eta)=O(|\eta|^{\de_1})$, where $\de_1>0$. 
Therefore, the integrand on the RHS of \eq{phip11} admits an upper bound via $Cg(|\xi|, |\eta|)$, where
\[
g(|\xi|,|\eta|)=\frac{|\xi|\min\{|\eta|^{\de_1}, |\eta|^{-\de}\}}{|\eta|(|\xi|+|\eta|)}.
\]
Evidently, there exists $C>0$ such that, for any $|\xi|$,  
$
\int_0^{+\infty}g(|\xi|,|\eta|)d|\eta|\le C,
$
hence, $\phi^{+,1}_q(\xi)$ is uniformly bounded in $\xi\in \cC_\ga$. This proves (2) for $\phipq(\xi)$; the proof
for $\phimq(\xi)$ is by symmetry.


(3a) Represent $\phipq(\xi)$ in the form $(q+\psi^0_{st}(\xi))^{-1}\phi^{+,1}_q(\xi)$, and set
$\phi^{-,1}_q(\xi)=\phimq(\xi)$, $\Phi(q,\eta)=\ln(1-i\mu\eta/(q+\psi^0_{st}(\eta))$. The same argument as in the proof of Lemma \ref{lem:stableWHF} gives
\eq{phip11} and \eq{phim11}, and the argument in the proof of (2) proves that $\phi^{\pm,1}_q(\xi)$ are uniformly bounded on the coni of interest.
Hence, \eq{asphipqp12} holds. To prove \eq{asphimqp12}, we substitute $\xi/(\eta(\xi-\eta))=1/\eta+1/(\xi-\eta)$ in
\eq{phim11}, represent the integral as the sum of two integrals, and, similarly to the proof of (3), prove that the exponential of the second integral is $O(|\xi|^{-\de})$, for some $\de>0$.

(3b) The proof is by symmetry.

(3c)-(3d) The proof is an evident modification of the proofs of (3b)-(3c).

(3e) Represent $\phipq(\xi)$ in the form $(1+i(\be\sg/q)\xi\ln\xi)^{-1}\phi^{+,1}_q(\xi)$, and set
$\phi^{-,1}_q(\xi)=\phimq(\xi)$, $\Phi(q,\eta)=\ln(1+\psi_{st}(\eta)/(1+i(\be\sg/q)\eta\ln\eta))$.
We have $\Phi(q,\eta)=1+O(1/\ln(|\eta|))$ as $\eta\to\infty$ along $\cL^-_\om$. Therefore, 
the same argument as in the proof of Lemma \ref{lem:stableWHF} gives \eq{phip11}. The integrand
on the RHS of \eq{phip11} is uniformly integrable over $\cL^-_{\om,0}=\{\eta\in \cL^-_\om \ |\ |\eta|\le 2\}$, and admits a bound via
$
g(|\xi|,|\eta|)=|\xi|(|\eta|(|\xi|+|\eta|)\ln|\eta|)^{-1}$ on $\cL^-_\om\setminus \cL^-_{\om,0}$. We have
\[
\int_2^{|\xi|}\frac{|\xi|d|\eta|}{|\eta|(|\xi|+|\eta|)\ln|\eta|}\le C\int_2^{|\xi|}|\eta|^{-1}\frac{d|\eta|}{\ln|\eta|}\le C_1\ln\ln|\xi|,
\]
and
\[
\int_{|\xi|}^\infty\frac{|\xi|d|\eta|}{|\eta|(|\xi|+|\eta|)\ln|\eta|}\le C_2|\xi|\int_{|\xi|}^{\infty}\frac{d|\eta|}{|\eta|^{2}\ln|\eta|}\le C_3,
\]
where $C,C_1,C_2,C_3$ are independent of $\xi$, therefore,
$\ln\phi^{+,1}_q(\xi)=O(\ln\ln|\xi|)$ as $\xi\to\infty$. We conclude that for any $\eps>0$, there exist $c,C>0$ such that
\[
c|\xi|^{1-\eps}\le |\phipq(\xi)\le C|\xi|^{1+\eps}
\]
if $|\xi$ is sufficiently large. The upper bound implies \eq{asphipq1p}, and the lower bound implies
\eq{asphimq1p} because $\psi_{st}(\xi)\sim i\sg\be |\xi|\ln|\xi|$ as $\xi\to \infty$.
 
 (3f) The proof is by symmetry.
 

 \subsection{Gaver-Wynn Rho algorithm}\label{GavWynn} 
  The inverse Laplace transform $V(T)$  of $\tilde V$  is approximated by 
\begin{equation}\label{GS31}
V(T, M)=\frac{\ln(2)}{t}\sum_{k=1}^{2M}\zeta_k \tV\left(\frac{k\ln(2)}{T}\right),
\end{equation}
where $M\in \bN$,
\begin{equation}\label{GS32}
\zeta_k(t, M)=(-1)^{M+k}\sum_{j=\lfloor (k+1)/2\rfloor}^{\min\{k,M\}}\frac{j^{M+1}}{M!}\left(\begin{array}{c} M \\ j\end{array}\right)
\left(\begin{array}{c} 2j \\ j\end{array}\right)\left(\begin{array}{c} j \\ k-j\end{array}\right)
\end{equation}
and $\lfloor a \rfloor$ denotes the largest integer that is less than or equal to  $a$.
If $T$ is large which in applications to option pricing means options of long maturities, then $q=k\ln(2)/T$ is small. In the present paper, efficient calculations of $\tV(f;q,x_1,x_2)$ are possible if $q\ge \sg$, where $\sg>0$ is determined by the parameters of the process and payoff function. Hence, if $T$ is large, we modify \eq{tVBrom}
\bbe\label{tVBrom_a}
V(f;T;x_1,x_2)=\frac{e^{aT}}{2\pi i}\int_{\Re q=\sg}e^{qT}\tV(f;q+a;x_1,x_2)\,dq,
\ee
where $a>0$ is chosen 
so that $\ln(2)/T+a>\max\{-\psi(i\mumpr),-\psi(i\muppr)\}$.
 In the paper, as in \cite{paired,Contrarian}, we apply Gaver-Wynn-Rho (GWR) algorithm, which is
  more stable than the Gaver-Stehfest method.
  
  Given a converging sequence $\{f_1, f_2,
\ldots\}$, Wynn's algorithm estimates the limit $f=\lim_{n\to\infty}f_n$ via $\rho^1_{N-1}$, where $N$ is even,
and $\rho^j_k$, $k=-1,0,1,\ldots, N$, $j=1,2,\ldots, N-k+1$, are calculated recursively as follows:
\begin{enumerate}[(i)]
\item
$\rho^j_{-1}=0,\ 1\le j\le N;$
\item
$\rho^j_{0}=f_j,\ 1\le j\le N;$
\item
in the double cycle w.r.t. $k=1,2,\ldots,N$, $j=1,2,\ldots, N-k+1$, calculate
\[
\rho^j_{k}=\rho^{j+1}_{k-2}+k/(\rho^{j+1}_{k-1}-\rho^{j}_{k-1}).
\]
We apply Wynn's algorithm with the Gaver functionals
\[
f_j(T)=\frac{j\ln 2}{T}\left(\frac{2j}{j}\right)\sum_{\ell=0}^j (-1)^j\left(\frac{j}{\ell}\right)\tilde f((j+\ell)\ln 2/T).
\]
 \end{enumerate}
 
 \section{Tables}\label{s:tables} 
\begin{table}
\caption{\small Cpdf  $V(T,a)=\bQ[\barX_T\le a\ |\ X_0=\barX_0=0]$, errors (rounded)
and CPU time. Parameters:  index $\al=1.2$, $a\in [0.0125, 0.075]$ are close to 0, the asymmetry is moderate: $\be=-0.2$, 
 $\mu=-0.02$; $T=0.25$.
 }
 {\tiny
\begin{tabular}{c|cccccc}\hline
 $a$ &  0.0125 &	0.025 &	0.0375 &	0.05 & 	0.0625 &	0.075 \\
 $V$ & 0.13205969881037
  & 0.238098430142687 & 0.339453622131327 & 0.435754264935413 &0.524541403377567 &
0.603375861525033
\\\hline
$A$ & 1.94E-11 &	4.22E-11 &	4.37E-11 &	4.50E-11 &	4.59E-11 &	4.64E-11\\
 $B$ & 1.06E-07 &	1.06E-07 &	1.07E-07 &	1.06E-07	& 1.06E-07	&1.06E-07\\ 
$C$ & 2.59E-06 &	2.27E-05	& 2.78E-05 &	-1.51E-05	& -3.42E-05 &	-1.88E-05
\\\hline
\end{tabular}
}

\begin{flushleft}{\tiny CPU time for 6 points, average over 1000 runs.\\
 Benchmark: SINH applied to the Bromwich integral, $N_\ell=414,$
$N^+_\pm=1106$, $N^-_\pm=826$. Errors $<E-15$. CPU time  362 msec. \\
A: SINH applied to the Bromwich integral. $N_\ell=174,$
$N^-_\pm=191$, $N^+_\pm=158$. CPU time 32.4 msec.
\\
B: SINH applied to the Bromwich integral. $N_\ell=125,$
$N^-_\pm=90$, $N^+_\pm=81$. CPU time 13.8 msec.\\
 C: GWR applied to the Bromwich integral, with $2M=16$, $N^-_\pm=1210$, $N^+_\pm=469$. CPU time 42.6 msec.\\
  NB: with much finer and longer grids, the errors of GWR are of the same order of magnitude.\\}
\end{flushleft}

\label{barX,al=1.2_mu=-0.02}
\end{table}

\begin{table}
\caption{\small Cpdf  $V(T,a)=\bQ[\barX_T\le a\ |\ X_0=\barX_0=0]$, errors (rounded)
and CPU time. Parameters:  index $\al=0.2$ and $a\in [0.0125, 0.075]$ are close to 0, the asymmetry is moderate: $\be=-0.2$, 
 $\mu=0$; $T=0.25$.
 }
 {\tiny
\begin{tabular}{c|cccccc}\hline
 $a$ &  0.0125 &	0.025 &	0.0375 &	0.05 & 	0.0625 &	0.075 \\
 $V$ & 0.86237781819448 & 0.878170612170767 & 0.88666389300912 & 0.89237116393858 & 0.896621627560969 &
0.899982988799815
\\\hline
$A$ & 1.30E-12 &	6.40E-13 &	5.56E-13 &	3.46E-13	& 3.47 E-13 &	3.270E-13\\
 $B$ & 1.47E-09 &	1.93E-09 &	-1.95E-09	& 7.51E-10 &	-2.01E-09 & 	-1.50E-09\\
$C$ & -1.27E-08 &	-8.57E-11 &	1.27E-08 &	2.23E-09 &	9.61E-09	& 5.13412E-10\\\hline
\end{tabular}
}

\begin{flushleft}{\tiny CPU time for 6 points, average over 1000 runs.\\
 Benchmark: SINH applied to the Bromwich integral. $N_\ell=185,$
$N^+_\pm=3910$, $N^-_\pm=705$. Errors $<E-15$. CPU time  969 msec. \\
A: SINH applied to the Bromwich integral. $N_\ell=125,$
$N^-_\pm=751$, $N^+_\pm=158$. CPU time 86.1 msec.
\\
B: SINH applied to the Bromwich integral.  $N_\ell=79,$
$N^-_\pm=718$, $N^+_\pm=135$. CPU time 51.2 msec.\\
 C: GWR applied to the Bromwich integral, with $2M=16$, $N^-_\pm=1210$, $N^+_\pm=469$. CPU time 42.6 msec.\\
 NB: with much finer and longer grids, errors of GWR are of the same order of magnitude.\\}
\end{flushleft}

\label{barX,al=0.2_mu=0}
\end{table}

\begin{table}
\caption{\small Cpdf  $V_\mu(T,a)=\bQ[\barX_T\le a\ |\ X_0=\barX_0=0]$, errors (rounded)
and CPU time. Parameters:  $\mu=\pm 0.02$, index $\al=0.2$ and $a\in [0.0125, 0.075]$ are close to 0, the asymmetry is moderate: $\be=-0.2$; 
 $T=0.25$.
 }
 {\tiny
\begin{tabular}{c|cccccc}\hline
 $a$ &  0.0125 &	0.025 &	0.0375 &	0.05 & 	0.0625 &	0.075 \\
 $V_{-0.02}$ & 0.86655235942346 &  0.880193992524779 & 0.887966922215758 &  0.893319578600464
 & 0.897361066299498 & 0.900585519405344
\\\hline
$V_{0.02}$ & 0.383040817220745 & 0.870818882637881 & 0.882207367996254 & 0.889280791200528 &
0.894277943437897 & 0.898108725916215
\\\hline
\end{tabular}
}

\begin{flushleft}{\tiny CPU time for 6 points, average over 1000 runs.
Since $\al<1$ and $\mu\neq 0$, SINH is not applicable. \\
Differences between the results obtain with GWR ($2M=16$) and different arrays are in the range
$[7\cdot 10^{-10}, 4\cdot 10^{-9}]$ for $V_{-0.02}$, and $[5\cdot 10^{-10}, 10^{-6}]$ for $V_{0.02}$.\\
$V_{-0.02}$ is calculated with $2M=16$, $N^-_\pm=827$, $N^+_\pm=311$; CPU time 64.1 msec.\\
$V_{0.02}$ is calculated with $2M=16$, $N^-_\pm=1983$, $N^+_\pm=518$; CPU time 75.5 msec.}
\end{flushleft}

\label{barX,al=0.2_mu=pm 0.02}
\end{table}

\begin{table}
\caption{\small Cpdf  $V_\mu(T,a)=\bQ[\barX_T\le a\ |\ X_0=\barX_0=0]$, errors (rounded)
and CPU time. Parameters:  $\mu=\pm 0.02$, index $\al=0.2$ and $a\in [0.0125, 0.075]$ are close to 0, the asymmetry is moderate: $\be=-0.2$;
 $T=0.004$ is small.
 }
 {\tiny
\begin{tabular}{c|cccccc}\hline
 $a$ &  0.0125 &	0.025 &	0.0375 &	0.05 & 	0.0625 &	0.075 \\
 $V_{-0.02}$ & 0.997491210718175 & 0.997814652512808 & 0.997984358113175 & 0.998096768145159
 & 0.998179648593157 & 0.998244693044529
\\\hline
$V_{0.02}$ & 0.997491211642771 & 0.99781465327481 & 0.997984358783901 & 0.99809676872475 &
0.998179649128261 & 0.998244693506221
\\\hline
\end{tabular}
}

\begin{flushleft}{\tiny CPU time for 6 points, average over 1000 runs.
Since $\al<1$ and $\mu\neq 0$, SINH is not applicable. \\
Differences between the results obtain with GWR ($2M=16$) and different arrays are in the range
$[4\cdot 10^{-10}, 7\cdot 10^{-10}]$ for $V_{-0.02}$, and $[3\cdot 10^{-9}, 2\cdot 10^{-8}]$ for $V_{0.02}$.\\
$V_{-0.02}$ is calculated with $2M=16$, $N^-_\pm=1546$, $N^+_\pm=683$; CPU time 62.5 msec.\\
$V_{0.02}$ is calculated with $2M=16$, $N^-_\pm=1546$, $N^+_\pm=683$; CPU time 63.3 msec.}
\end{flushleft}

\label{barX,al=0.2_mu=pm 0.02,T=0.004}
\end{table}

\begin{table}
\caption{\small Cpdf  $V_\mu(T,a)=\bQ[\barX_T\le a\ |\ X_0=\barX_0=0]$, errors (rounded)
and CPU time. Far in the tail: $a\in [100,600]$, $T=1$ is moderately large.
Other 
parameters:  $\mu=\pm 0.02$, index $\al=0.2$ is close to 0, the asymmetry is moderate: $\be=-0.2$. 
 }
 {\tiny
\begin{tabular}{c|cccccc}\hline
 $a$ & 100 &	200 &	300 &	400 & 	500 &	600 \\
 $V_{-0.02}$ & 0.90467837424503 & 0.916070201834554 & 0.922144074014908 &
 0.926205338569077 & 0.929219523163341 & 0.931596967461542 
\\\hline
$V_{0.02}$ & 0.904671590132716 & 0.916067179091704 & 0.922142699482402 &
0.926203995172241 &  0.929218489731696 & 0.931596132973088
\\\hline
\end{tabular}
}

\begin{flushleft}{\tiny CPU time for 6 points, average over 1000 runs.
Since $\al<1$ and $\mu\neq 0$, SINH is not applicable. \\
Differences between the results obtain with GWR ($2M=16$) and different arrays are in the range
$[2\cdot 10^{-11}, 2\cdot 10^{-10}]$ for $V_{-0.02}$, and $[7\cdot 10^{-11}, 9\cdot 10^{-10}]$ for $V_{0.02}$.\\
$V_{-0.02}$ is calculated with $2M=16$, $N^-_\pm=1579$, $N^+_\pm=428$; CPU time 79.1 msec.\\
$V_{0.02}$ is calculated with $2M=16$, $N^-_\pm=1865$, $N^+_\pm=466$; CPU time 128.0 msec.}
\end{flushleft}

\label{barX,al=0.2_mu=pm 0.02,T=1}
\end{table}

\begin{table}
\caption{\small Cpdf  $V(T,a)=\bQ[\barX_T\le a\ |\ X_0=\barX_0=0]$, errors (rounded)
and CPU time. Parameters:  small index $\al=0.2$, time to maturity $T=10$ is large,  $a\in [0.5, 3.0]$ are moderately large, the asymmetry is moderate: $\be=-0.2$, 
$\mu=0$; $T=0.25$.
 }
 {\tiny
\begin{tabular}{c|cccccc}\hline
 $a$ &  0.5 &	1.0 &	1.5 &	 2.0 & 	2.5 &3.0 \\
 $V$ & 0.31071084287788 & 0.32976167923584 & 0.341569702301114 &
 0.3502636252052 & 0.357194432515148 & 0.362981755403084
 \\\hline
$A$ & 1.49E-10 &	1.22E-10 &	1.22E-10 &	1.20E-10 &	9.97E-11 &	9.68E-11
\\
 $B$ & 1.49E-08 &	3.45E-08 &	-4.20E-08 &	-1.90E-08	& -4.30E-08 &	-1.28E-08\\
$C$ & -1.27E-08 &	-8.57E-11 &	1.27E-08 &	2.23E-09 &	9.61E-09	& 5.13412E-10\\\hline
\end{tabular}
}

\begin{flushleft}{\tiny CPU time for 6 points, average over 1000 runs.\\
 Benchmark: SINH applied to the Bromwich integral. $N_\ell=310,$
$N^+_\pm=6301$, $N^-_\pm=1376$. Errors $<E-15$. CPU time  969 msec. \\
A: SINH applied to the Bromwich integral. $N_\ell=125,$
$N^-_\pm=1074$, $N^+_\pm=177$. CPU time 105.7 msec.
\\
B: SINH applied to the Bromwich integral. $N_\ell=79,$
$N^-_\pm=455$, $N^+_\pm=69$. CPU time 36.5 msec.\\
 C: GWR applied to the Bromwich integral, with $2M=16$, $N^-_\pm=2021$, $N^+_\pm=435$. CPU time 125.9 msec.\\
 NB: with much finer and longer grids, the errors of GWR are of the same order of magnitude.\\}
\end{flushleft}

\label{barX,al=0.2_mu=0,T=10}
\end{table}

 \begin{table}
\caption{\small Joint CPDF. Values  of $V_1(T,a_1,a_2)=\bQ[X_T\le a_1\ |\ X_0=0]-\bQ[X_T\le a_1, \barX_T\le a_2\ |\ X_0=\barX_0=0]
 $, 
 errors (rounded) of two numerical schemes
and CPU time. Parameters:  index $\al=1.2$, $a_2\in [0.0125, 0.075]$ and $a_{12}:=a_1-a_2\in [-0.075,-0.0125]$ are close to 0, the asymmetry is moderate: $\be=-0.2$, 
 $\mu=-0.02$; $T=0.25$.
 }
 {\tiny
\begin{tabular}{c|cccccc}
\hline\hline
$a_2/a_{12}$ & -0.075 &  -0.05 & -0.0375 & -0.025 & -0.0125 \\\hline
0.0125 & 0.12244233311163  & 0.157907371799232 & 0.181683980001225
& 0.210889139495737 & 0.246912884013257 \\
0.025 & 0.0918916219424353  & 0.120389641188601 & 0.140202136800179 &
0.165401766901898 & 0.197967194459645\\
0.0375 & 0.0692609560212789  & 0.0918389965351731 & 0.108011387880229 &
0.129210336122806 & 0.157781151351146\\
0.05 & 0.0521567547065453  & 0.06974459359541 & 0.0826439620684098 &
0.099983838478999 & 0.124230282077777\\
0.0625 & 0.0393164899603866  & 0.0528386851043925 & 0.0629314657230151
& 0.0767702363595711 & 0.0967218956392104\\
0.075 & 0.0297971711410534  & 0.0401248577840783 & 0.0479246215441564
& 0.0587741609274394 & 0.0747900267865678
\\\hline
A &  & & Errors of SINH & &\\\hline
$a_2/a_{12}$ & -0.075 & - -0.05 & -0.0375 & -0.025 & -0.0175 \\\hline
0.0125 & -3.00-09 &	 -3.14E-09 &-3.22E-09 &	-3.31E-09 &	-3.33E-09\\
0.025 & -2.54E-09 &	 -2.68E-09 &	-2.78E-09	& -2.90E-09 &	-3.05E-09\\
0.0375 & -2.14E-09 &		-2.26E-09	& -2.35E-09 &	-2.46E-09 &	-2.61E-09\\
0.05 & -1.79E-09 &		-1.88E-09	&-1.96E-09 &	-2.06E-09 &	-2.21E-09\\
0.0625 & -1.48E-09 &	 	-1.55E-09 & 	-1.61E-09	& -1.70E-09 &	-1.84E-09\\
0.075 & -1.21E-09 &	-1.27E-09 &	-1.32E-09 &	-1.39E-09 &	-1.51E-09\\\hline

B &  & & Errors of GWR & &\\\hline
$a_2/a_{12}$ & -0.075 &  -0.05 & -0.0375 & -0.025 & -0.0175 \\\hline
0.0125 & 2.92E-08 &		5.84E-07 &	8.40E-07	& -5.81E-07 &	-2.76E-06\\
0.025 & 2.06E-07 &		8.82E-07 &	2.08E-06 &	2.56E-06 &	-2.60E-06\\
0.0375 & 2.65E-06 & 		2.50E-06 &	3.52E-06 &	6.11E-06 &	8.60E-06\\
0.05 & -2.27E-06 &	 	3.92E-05	& 7.80E-06 &	6.60E-06	& 9.70E-06\\
0.0625 & -8.35E-07 &	 -2.71E-06 &	-7.44E-06	& 4.58E-04 &1.06E-05\\
0.075 & 1.66E-07 &		-6.19E-08 &	-9.73E-07 &	-3.81E-06	&-1.85E-05\\\hline

\end{tabular}

}
\begin{flushleft}{\tiny
Errors of the benchmark values: better than E-15. CPU time for one point, average over 1000 runs: 1,982 msec.\\

A: SINH applied to the Bromwich integral. $N_\ell=193, N^-_\pm=144, N^+_\pm=147$. CPU time  for one point, average over 1000 runs: 421 msec.\\
B:  Gaver-Wynn
Rho algorithm, $2M=16$, $N^-_\pm=516, N^+_\pm=677$. CPU time  for one point, average over 1000 runs:  194 msec.\\}
\end{flushleft}

\label{joint_al=1.2,mu=0,T=0.25}
 \end{table}

\begin{table}
\caption{\small Joint CPDF. Values  of $V_1(T,a_1,a_2)=\bQ[X_T\le a_1\ |\ X_0=0]-\bQ[X_T\le a_1, \barX_T\le a_2\ |\ X_0=\barX_0=0]
 $, 
 errors (rounded) of two numerical schemes
and CPU time. Parameters:  index $\al=0.2$, $a_2\in [0.0125, 0.075]$ and $a_{12}:=a_1-a_2\in [-0.075,-0.015]$ are close to 0, the asymmetry is moderate: $\be=-0.2$, 
$\mu=0$; $T=0.25$.
 }
 {\tiny
\begin{tabular}{c|cccccc}
\hline\hline
$a_2/a_{12}$ & -0.075 &  -0.05 & -0.0375 & -0.025 & -0.0125 \\\hline
0.0125 & 0.00682358422697427 & 0.00705249804952593 & 0.00720455195463735
& 0.00740278409139612 & 0.00769492446912139\\
0.025 & 0.00558447991435967 & 0.00574082324058917 & 0.00584237670326238 &
0.00597177710600839 & 0.00615563177428208\\
0.0375 & 0.00494072158394288 & 0.00506309108194094 & 0.0051414907517235 & 
0.00524002286520389 & 0.00537705594316782\\
0.05 & 0.00451892651994965 & 0.00462072587777185& 0.00468530333674396 &
0.00476567446164324 & 0.00487579384718796 \\
0.0625 & 0.0042111823310130 & 0.004298919274522 & 0.00435414844090055 &
0.00442236983151984 & 0.00451478766166025\\
0.075 & 0.0039720388807442 &0.00404944263994188 & 0.00409786185433674 &
0.00415730883426692 & 0.00423711210468661
\\\hline
A &  & & Errors of SINH & &\\\hline
$a_2/a_{12}$ & -0.075 & - -0.05 & -0.0375 & -0.025 & -0.0175 \\\hline
0.0125 & -4.02E-09 &		-4.24E-09 &	-4.40E-09 &	-4.65E-09	 & -5.10E-09\\

0.025 & -2.42E-09 &		-2.52E-09 &	-2.60E-09	 &-2.71E-09 &	-2.93E-09\\

0.0375 & -1.85E-09	& 	-1.912E-09 &	-1.97E-09	 &-2.04E-09 &	-2.18E-09\\

0.05 & -1.49E-09 &	 	-1.53E-09	& -1.56E-09 &	-1.612E-09 &	-1.70E-09\\

0.0625 & -1.29E-09 &		-1.32E-09 &	-1.34E-09 &	-1.37E-09	& -1.43E-09\\

0.075 & -1.21E-09 &		-1.24E-09 &	-1.26E-09	&-1.29E-09 &	-1.34E-09\\\hline

B &  & & Errors of GWR & &\\\hline
$a_2/a_{12}$ & -0.075 &  -0.05 & -0.0375 & -0.025 & -0.0175 \\\hline
0.0125 & -1.78E-08&	-2.21E-08	&-1.49E-08 &	-7.99E-09 &	-3.21E-08\\

0.025 & -1.34E-08&		-2.52E-08	& -2.32E-08 &	-2.51E-08	& -1.51E-08\\

0.0375 & -1.90E-08	&	-1.36E-08	& -2.31E-08 &	-2.92E-08	& -1.93E-08\\

0.05 & -1.09E-08 &	 -1.04E-08 &	-1.39E-08	& -1.83E-08 &	-1.14E-08\\

0.0625 & -1.41E-08 &	-1.03E-08	& -5.86E-09 &	-1.18E-08	& -5.90E-09\\

 0.075 &-1.46E-08&	-8.92E-09 &	-1.37E-08 &	-1.39E-08	& -1.22E-08 \\\hline

\end{tabular}

}
\begin{flushleft}{\tiny
Errors of the benchmark values: better than E-15. CPU time for one point, average over 1000 runs: 8.223 msec.\\

A: SINH applied to the Bromwich integral. $N_\ell=86, N^-_\pm=373, N^+_\pm=65$. CPU time  for one point, average over 1000 runs: 254 msec.\\
B:  Gaver-Wynn
Rho algorithm, $2M=16$, $N^-_\pm=3028, N^+_\pm=632$. CPU time  for one point, average over 1000 runs:  987 msec.\\}
\end{flushleft}

\label{joint_al=0.2,mu=0,T=0.25}
 \end{table}

\end{document}